\documentclass[10pt]{amsart} 

\usepackage[psamsfonts]{amssymb} 
\usepackage{amsfonts,amsmath} 
\usepackage{graphicx, color}
\usepackage{enumerate}
\usepackage{url}

\newtheorem{thmA}{Theorem}

\newtheorem{theorem}{Theorem}[section]

\newtheorem{proposition}[theorem]{Proposition}
\newtheorem{lemma}[theorem]{Lemma}
\newtheorem{corollary}[theorem] {Corollary}

\newtheorem{question}[theorem]{Question}

\theoremstyle{remark}
\newtheorem{remark}[theorem]{Remark}
\newtheorem{example}[theorem]{Example}

\theoremstyle{definition}
\newtheorem{definition}[theorem]{Definition}

\def\Z{\mathbb Z}
\def\R{\mathbb R}

\def\G{\Gamma}

\def\aut{{\rm{Aut}}}
\def\out{{\rm{Out}}}

\def\hom{{\rm{Hom}}}

\def\<{\langle}
\def\>{\rangle}
\def\psa{{\rm{PSA}}(A_{\G})}
\def\pso{{\rm{PSO}}(A_\G)}

\newcommand{\st}{\mathrm{st}}
\newcommand{\lk}{\mathrm{lk}}
\newcommand{\supp}{\mathrm{supp}}
\newcommand{\vspan}{\mathrm{span}}
\newcommand{\zg}[2]{\zeta^{#1}_{#2}}
\newcommand{\eg}[2]{\eta^{#1}_{#2}}

\title[Subspace arrangements, BNS invariants, and $\pso$]{Subspace arrangements, BNS invariants, and pure symmetric outer automorphisms of right-angled Artin groups}
\author{Matthew B. Day and Richard D. Wade}   
\date{28 June, 2016}
\begin{document}
\begin{abstract} 
We introduce a homology theory for subspace arrangements, and use it to extract a new system of numerical invariants from the Bieri-Neumann-Strebel invariant of a group. 
We use these to characterize when the set of basis conjugating outer automorphisms (a.k.a. the pure symmetric outer automorphism group) of a right-angled Artin group is itself a right-angled Artin group. 
\end{abstract}

\address{}
\email{}

\maketitle

\section{Introduction}
\subsection{Motivation}
Recall that a \emph{right-angled Artin group (RAAG)} is a group $A_\G$ given by a finite presentation whose only relations are that some pairs of generators commute (see Section~\ref{ss:RAAG} below).
Outer automorphism groups of RAAGs form a diverse and interesting family of groups.
We are  motivated by the following question:
\begin{question}\label{question:motivating}
When does the outer automorphism group $\out(A_\Gamma)$ contain another RAAG as subgroup of finite index?
What combinatorial conditions on the defining graph $\Gamma$ characterize this?
\end{question}
We feel that this is an important test question in terms of the field's understanding of these groups.
Two important sequences of outer automorphism groups of RAAGs, $\out(F_n)$ and $\rm{GL}_n(\mathbb{Z})$, exhibit very different behavior when $n=2$ compared to when $n\geq 3$.
In the case $n=2$, both are virtually free, but if $n\geq 3$ neither is virtually a RAAG (see references below).
The idea is to identify a `low rank' or `low complexity' type for the family of outer automorphism groups of RAAGs. There are variants of this question where $\out(A_\Gamma)$ is replaced by the automorphism group $\aut(A_\Gamma)$, or where instead of asking about finite-index subgroups, we ask more generally when $\out(A_\Gamma)$ is commensurable to a RAAG.

As well as the above virtually free examples, there are examples of RAAGs whose outer automorphism groups are finite \cite{CF, D} or virtually free abelian \cite{BF}. There are also some more interesting examples; for instance $\out(F_2 \times F_2)$ is commensurable with $F_2 \times F_2$ itself. On the other side of the spectrum:

\begin{itemize}
\item $\out(A_\G)$ may contain distorted cyclic subgroups (e.g. $\rm{GL}_3(\mathbb{Z})$).
\item $\out(A_\G)$ may contain a poison subgroup, forcing nonlinearity \cite{FP,AMP}.
\item $\out(A_\G)$ may have an exponential Dehn function (e.g. $\out(F_3)$ \cite{BV}).
\end{itemize}

All of these phenomena rule out the possibility of finite index subgroups being RAAGs. Despite these tools, a complete answer to Question~\ref{question:motivating} seems difficult. It is often tricky to tell whether a group is a RAAG on the nose or not, let alone up to finite index. For example, the group

\[G=\langle a,b,c,d,e | [a,b], [c,d], [ab,c], [cd,a]\rangle\] 
is a nonstandard presentation of the RAAG $(F_2 \times F_2)\ast \mathbb{Z}$, but after adding the innocent-looking relations $[e,b]$ and $[e,d]$, the group 
\[G'=\langle a,b,c,d,e | [a,b], [c,d], [ab,c], [cd,a], [e,b], [e,d]\rangle\]
is not isomorphic to a RAAG (this can be shown using the methods in this paper). This leads us to:
\begin{question}\label{question:recognizeRAAGs}
Suppose $G$ is a group given by a finite presentation whose only relations are commutators (between words in the generators).
Is there a procedure to recognize if $G$ is a RAAG?
\end{question}
This question is stated so generally that the answer is almost certainly `no', but for specific classes of groups the question is still interesting.
To show such a group is a RAAG, we need some kind of rewriting procedure for the presentation, and to show it is not a RAAG, we usually need some kind of subtle invariant.
One such invariant is the BNS invariant; Koban and Piggott used the BNS invariant to distinguish the non-RAAGs from a certain class of groups in a recent paper~\cite{KP}.
We discuss this below.  

\subsection{BNS invariants}
The BNS invariant $\Sigma$ of a finitely generated group $G$ was introduced in \cite{BNS}. 
It is an open subset of the character sphere of $G$ (i.e. the unit sphere of $\hom(G;\R)$) and it records the existence of certain kinds of actions on $\R$--trees.
We review the BNS invariant in Section~\ref{ss:bns} below. There is a growing collection of groups for which there is an explicit description of $\Sigma$.
These examples include:
\begin{itemize}
\item Fundamental groups of compact 3--manifolds \cite{BNS,T}.
\item Right-angled Artin groups \cite{MVW}.
\item Pure braid groups \cite{KMM}.
\item Pure symmetric automorphisms of right-angled Artin groups \cite{KP}.
\item Many hierarchies of groups over groups with trivial BNS invariants (see \cite{CL} for a precise formulation).
\end{itemize}

In the above examples, the complement $\Sigma^c$ is a union of linear subspheres of the character sphere, so that the pre-image of $\Sigma^c$ in $\hom(G;\mathbb{R})$ determines a set of subspaces $\mathcal{V}_G$ of $\hom(G;\mathbb{R})$. 
For an arbitrary pair $(V,\mathcal{V})$ consisting of a vector space $V$ and a set of subspaces $\mathcal{V}$ of $V$, one can define a chain complex $C_*(V,\mathcal{V})$ where $C_0=V$ and each $C_n$ is a formal direct sum of intersections of $n$ subspaces in $\mathcal{V}$. We describe this chain complex in Section~\ref{s:cc}, although we would be interested to know if it has appeared in the literature previously.   This chain complex has associated homology spaces $H_*(V,\mathcal{V})$. One can then study the homology \[H_*(\mathcal{V}_G)= H_*(\hom(G;\R),\mathcal{V}_G) \]
given by the arrangement of maximal complementary subspaces $\mathcal{V}_G$ in $\hom(G;\R)$. 
In the above list of examples, $\mathcal{V}_G$ is a finite set of subspaces of $\hom(G;\R)$, which allows for $H_*(\mathcal{V}_G)$ to be computed explicitly. 

In general, one can still define $\mathcal{V}_G$ to be the subspace arrangement consisting of maximal subspaces $V \subset \hom(G;\R)$ such that the equivalence class $[\chi]$ of each nontrivial  $\chi \in V$ is contained in $\Sigma^c$. This subspace arrangement only contains every character in the complement of the BNS invariant when  $\Sigma$ is \emph{symmetric} in the character sphere (i.e. $\Sigma=-\Sigma$). Nevertheless, the Betti numbers for this homology theory still provide a concrete set of numerical invariants for an arbitrary group $G$.

The homology theory above is heavily influenced by a recent paper of Koban--Piggott \cite{KP}, who determine exactly when the \emph{pure symmetric automorphism group}  of $A_\Gamma$ is itself a RAAG. This is directly related to our Question~\ref{question:motivating} because there are many examples of RAAGs where the pure symmetric automorphisms form a finite-index subgroup of $\aut(A_\Gamma)$.
This pure symmetric automorphism group $\psa$, sometimes called the \emph{basis conjugating automorphism group} is the subgroup consisting of automorphisms that take each element of a graphical basis of $A_\G$ to a conjugate of itself. The group $\psa$ has a standard generating set where each generator $\pi_K^a$ is given by a vertex $a \in \G$ and a component of $K$ of $\G-\st(a)$ (here $\st(a)$ is the subgraph of $\G$ spanned by $a$ and its adjacent vertices). The generator $\pi_K^a$ acts on each vertex of $\G$ by:
\begin{equation*} \pi_K^a(x) = \begin{cases} axa^{-1}& \text{if $x \in K$} \\ x &\text{otherwise}\end{cases} \end{equation*}
Toinet \cite{Toinet} gave a presentation of $\psa$ which was simplified by Koban and Piggott to one that uses the above generators (see Theorem~\ref{t:psapresentation}). 
This presentation for $\psa$ is the standard presentation of a RAAG unless the graph $\G$ contains a \emph{separating intersection of links}, or SIL (often pronounced `sill').
 A SIL occurs when there is a common component $K$ of both $\G -\st(a)$ and $\G -\st(b)$ for two non-adjacent vertices $a$ and $b$. 
This `no SIL' RAAG presentation of $\psa$ first appeared in work of Charney et al. \cite{CRSV}. 

In the converse direction, Koban and Piggott give an explicit description of the BNS invariant $\Sigma(\psa)$ and show that its complement is a set of rationally defined linear subspheres of the character sphere. 
Furthermore, they find an invariant which allows them to prove that when the graph $\G$ contains a SIL, the group $\psa$ is not a RAAG. 
The invariant they use coincides with the Euler characteristic of $H_*(\mathcal{V}_G)$. 
In our terminology, their results state:

\begin{thmA}[Koban--Piggott, \cite{KP}] If $G$ is a right-angled Artin group then the Euler characteristic of $H_*(\mathcal{V}_G)$ is equal to the rank of the center of $G$ (in particular, it is non-negative). If $G=\psa$ then either:
 \begin{itemize}  \item the graph $\G$ does not contain a SIL and the Euler characteristic of $H_*(\mathcal{V}_G)$ is zero; therefore $G$ is a RAAG with trivial center; or 
\item the graph $\G$ contains a SIL and the Euler characteristic of $H_*(\mathcal{V}_G)$ is strictly negative; $G$ is not a RAAG.\end{itemize}
\end{thmA}

\subsection{Results}

In this paper, we study the image of $\psa$ in $\out(A_\G)$, which we call the \emph{pure symmetric outer automorphism group} of $A_\G$ and denote by $\pso$.  
We give a description of  $\Sigma^c(\pso)$ (see Proposition~\ref{p:bnspso}) as a finite set of rationally defined subspheres of the character sphere and classify when $\pso$ is itself a RAAG. 
Rather than being based on the existence of a SIL, this classification depends on \emph{how SILs are arranged in $\G$}. Let us describe this in a precise way: 
when $a$ and $b$ have a separating intersection of links, relations of the form $[\pi_K^a\pi^a_L,\pi^b_L]$ appear in the presentation of $\psa$, where $b \in K$ and $L$ is what we call a \emph{shared component} of both $\G - \st(a)$ and $\G - \st(b)$. The following graph gives a combinatorial description of when two components $K$ and $L$ of $\G - \st(a)$ occur in such a relation.

\begin{definition} For each vertex $a \in \G$ the \emph{support graph} $\Delta_a$ has a vertex for each component of $\G - \st(a)$. There is an edge between two components $K$ and $L$ if there exists a vertex $b$ such that $b \in K$ and $L$ is a shared component of both $\G - \st(a)$ and $\G -\st(b)$.
\end{definition}

In particular, the graph $\G$ has no SIL if and only if each support graph $\Delta_a$ is discrete. The following theorem is the main result of our paper. The first part describes $H_*(\mathcal{V}_G)$ precisely when $G$ is a RAAG, and the second part describes how the support graphs determine when $\pso$ is isomorphic to a RAAG.

\begin{thmA}\label{t:main} If $G$ is a right-angled Artin group then:
\begin{enumerate} 
\item $\dim(H_0(\mathcal{V}_G))$ is equal to the rank of the center of $G$.
\item $H_n(\mathcal{V}_G)=0$ if $n>0$.
\end{enumerate}
If $G=\pso$ then either:
\begin{itemize}
\item Each support graph $\Delta_a$ is a forest and $G$ is isomorphic to a right-angled Artin group.
\item For some vertex $a \in \G$ the support graph $\Delta_a$ contains a loop. Then $H_1(\mathcal{V}_G)$ is nontrivial and therefore $G$ is not a right-angled Artin group.
\end{itemize}
\end{thmA}

Our methods give effective algorithms to determine whether $\pso$ is a RAAG for a given $\G$, and to identify which RAAG it is, if it is one.
We encourage our readers to try out several examples, but we only give two here.
\begin{example}
Let $\G$ be the edgeless graph on three vertices $\{a,b,c\}$, so $A_\G$ is the free group $F_3$.
Koban--Piggott's theorem shows that $\rm{PSA}(F_3)$ is not a RAAG, because $a$ and $b$ form a SIL.
However, all three of the support graphs consist of a single edge, so are trees.
One can check that $\rm{PSO(F_3)}$ is a free group generated by the set $\{\pi^a_{b},\pi^b_{c},\pi^c_{a}\}$. 
\end{example}
\begin{example}
Let $\G$ be the edgeless graph on four vertices $\{a,b,c,d\}$, so $A_\G$ is the free group $F_4$.
Again, $\rm{PSA}(F_4)$ is not a RAAG because there are SILs.
All four of the support graphs are triangles; for example the path labeled by $\{b\}$--$\{c\}$--$\{d\}$--$\{b\}$ is a loop in $\Delta_a$.
We can use this loop to produce a nontrivial element of $H_1(\mathcal{V}_G)$, which implies that $G=\rm{PSO}(F_4)$ is also not a RAAG.
\end{example}

The paper is arranged as follows: Section~\ref{s:background} contains background material on right-angled Artin groups and their symmetric automorphisms. Section~\ref{s:homology} defines the homology $H_*(V,\mathcal{V})$ associated to a subspace arrangement $\mathcal{V}$ in a vector space $V$. It may be read independently from the rest of the paper. We describe some simple examples and show that $H_*$ is functorial with respect to morphisms between subspace arrangements.  In Section~\ref{s:bns} we apply this to BNS invariants of groups. We first give the general definition of $H_*(\mathcal{V}_G)$ before looking at the case when $G$ is equal to $A_\G$, $\psa$, or $\pso$ respectively. In particular we use Koban and Piggott's description of $\Sigma(\psa)$ to give a description of $\Sigma(\pso)$. Finally, in Section~\ref{s:presentation} we give an explicit RAAG presentation for $\pso$ when each support graph $\Delta_a$ is a forest. 
The main contribution in this final section is 
a description of a generating set for $\pso$ that serves as the standard basis for a graphical RAAG presentation (if there are SILs then the original generating set will not work).
This uses the structure of the support graphs in an essential way.

\subsection*{Acknowledgments}
The authors would like to thank Dawid Kielak, Lance Miller, Adam Piggott, Henry Schenck, Alex Suciu, and Uli Walther for helpful conversations.
They would also like to thank an anonymous referee for comments that improved the paper. 
Matthew Day was supported in part by NSF grant DMS-1206981.

\section{Pure symmetric automorphisms of RAAGs}\label{s:background}

\subsection{Right-angled Artin groups}\label{ss:RAAG}
A finite graph $\G$ with vertex set $V(\G)$ and edge set $E(\G)$ determines the \emph{right-angled Artin group} $A_\G$ with presentation:
\[ A_\G= \Big\langle V(\G) \Big| \big\{ [v,w] : \{v,w\}\in E(\G) \big\}\Big \rangle \] 

That is, the generators of $A_\G$ are the vertices of $\G$, and they commute if they are connected by an edge in $\G$. 
We call such a presentation a \emph{graphical presentation} for the RAAG. 
For $v \in \G$, its \emph{link} $\lk(v)$ is the 
set of
 vertices adjacent to $v$, and its \emph{star} $\st(v)$ is 
$\lk(v)\cup\{v\}$.
 For a word $w$, the \emph{support} of $\supp(w)$ consists of each vertex $v$ such that $v$ or $v^{-1}$ appears in $w$. 
A word $w$ is \emph{reduced} if we cannot cancel any inverse pairs of elements appearing in it: for any subword of the form $v^\epsilon w'v^{-\epsilon}$, the support of $w'$ is not contained in the star of $v$.
 The support of an element $g \in A_\G$ is the support of any reduced word representing $g$. 
This is independent of the reduced representative. 
For any full subgraph $\G'$, the group $A_{\G'}$ naturally embeds in $A_\G$ as the subgroup generated by the vertices in $\G'$, so $A_{\G'}=\langle \G'\rangle\subset A_\G$.  

For any vertex $v \in \G$, its centralizer $C(v)$ is the subgroup $A_{\st(v)}$.
This is an easy special case of Servatius's centralizer theorem~\cite{Servatius}.
The center $Z(A_\G)$ of $A_\G$ is the free abelian subgroup $A_{\G'}$, where $\G'$ is the span of the set of vertices adjacent to every other vertex in the graph.

\subsection{Symmetric automorphisms}

A \emph{partially symmetric automorphism} of $A_\G$ is an automorphism $\phi \in \aut(A_\G)$ such that each vertex $v \in A_\G$ is sent to a conjugate $gvg^{-1}$ under $\phi$. The conjugating element $g$ is allowed to vary with $v$. The set $\psa$ forms a subgroup of $\aut(A_\G)$. We define $\pso$ to be the image of $\psa$ in the outer automorphism group $\out(A_\G)$. If $\phi$ is an automorphism, we use $[\phi]$ to denote the equivalence class represented by $\phi$ in $\out(A_\G)$. Each vertex $a \in \G$ and component $K$ of $\G - \st(a)$ defines an automorphism $\pi_K^a$ of $\psa$, where:
\begin{equation*} \pi_K^a(x) = \begin{cases} axa^{-1}& \text{if $x \in K$} \\ x &\text{otherwise}\end{cases} \end{equation*}
We refer to elements of the form $\pi^a_K$ when $K$ is a component of $\G - \st(a)$ as \emph{standard generators} of $\psa$, and the set $X$ of all such elements as the \emph{standard generating set} of $\psa$. If $C=K_1 \cup K_2 \cup \cdots \cup K_n$ is a nontrivial union of connected components of $\G - \st(a)$, we may define $\pi^a_C$ in the same way as above. However, as \[ \pi_C^a=\pi_{K_1}^a\pi_{K_2}^a\cdots \pi_{K_n}^a\] we leave these elements out of our generating set $X$. We will refer to all elements of the form $\pi_C^a$ as \emph{partial conjugations} and reserve the term \emph{standard generator} for an element of the form $\pi_K^a$ when $K$ is a single connected component of $\G-\st(a)$. The element $a$ is called the \emph{multiplier} of the partial conjugation.  

\subsection{Commutation in $\out(A_\G)$}

The following lemma is a rephrasing of the classification of connected components given in \cite{MR2888948}. This classification is used throughout the paper, so for completeness we give a brief proof.

\begin{lemma}\label{l:components}
Let $a$ and $b$ be nonadjacent vertices of $\G$. We can write the components of $\G- \st(a)$ as $A_0,\ldots, A_k, C_1, \ldots, C_l$ and the components of $\G - \st(b)$ as $B_0,\ldots, B_m, C_1, \ldots, C_l$ where

\begin{itemize}
\item we have $b \in A_0$ and $a \in B_0$, and
\item $A_1, \ldots A_k \subset B_0$ and $B_1, \ldots, B_m \subset A_0$. 
\end{itemize}
\end{lemma}

We say that $A_0$ and $B_0$ are the \emph{dominating components}, that $A_1, \ldots A_k$ and $B_1,\ldots B_k$  are the \emph{subordinate components}, and $C_1, \ldots C_l$ are the \emph{shared components} for the pair $(a,b)$. We will sometimes use $[b]_a$ to denote the component of $\G -\st(a)$ containing $b$, i.e. $[b]_a$ is the dominating component of $\G-\st(a)$ with respect to $b$. 
Note that if we fix $a$ the roles of the connected components of $\G - \st(a)$ may change as we vary $b$. 

\begin{proof}[Proof of Lemma~\ref{l:components}]
Let $K$ be a component of $\G-\st(a)$ that is not the dominating component (so $b\notin K$).
It is enough to show that $K$ is either a subordinate component or a shared component, since by the symmetry between $a$ and $b$ this will show that all components fall into the classification.

First we note that $K\cap\st(b)=\varnothing$.
If this were not the case, there would be a path of length one from $b$ to an element of $K$, and since $b$ is not adjacent to $a$ (and $\st(a)\cap K=\varnothing$) this would imply that $b\in K$, counter to our hypothesis.
Since $K\cap\st(b)=\varnothing$, $K$ is a subset of a single component of $\G-\st(b)$ (every path in $K$ avoids $\st(b)$, so every path in $K$ is a path in $\G-\st(b)$).

We break into two cases: (1) there is an edge from $K$ to an element of $\lk(a)-\lk(b)$ and (2) every edge from $K$ to $\lk(a)$ connects to an element of $\lk(a)\cap\lk(b)$.
In case (1), there is a path from $K$ to $a$ avoiding $\st(b)$, by passing through an element of $\lk(a)-\lk(b)$.
This means that there is a path from every element of $K$ to $a$ avoiding $\st(b)$, so that $K$ is a subset of the dominating component of $\G-\st(b)$ with respect to $a$, and therefore $K$ is a subordinate component.
In case (2), every path starting in $K$ and avoiding $\st(b)$ must also avoid $\st(a)$, since every edge from $K$ to $\lk(a)$ must connect to an element of $\st(b)$.
This means that paths starting in $K$ that avoid $\st(b)$ cannot escape $K$; in other words the component of $\G-\st(b)$ containing $K$ does not contain any elements outside of $K$ and must equal $K$.
So in case (2), $K$ is a shared component.
\end{proof}

Guiterrez, Piggott and Ruane \cite{MR2888948} give the following definition to describe when there exist shared components for the pair $(a,b)$:

\begin{definition}
We say that a pair $(a,b)$ forms a \emph{separating intersection of links} or is an \emph{SIL-pair} if $a$ and $b$ are nonadjacent and there is a connected component $R$ of $\G - (\lk(a) \cap \lk(b))$ with $a,b \not\in R$.
\end{definition}

\begin{lemma}[\cite{MR2888948}, Lemma 4.5]
A pair $(a,b)$ is an SIL-pair if and only if the set of shared components associated to $(a,b)$ is nonempty.
\end{lemma}

\begin{proof}
From the above proof of Lemma~\ref{l:components}, a component $K$ is shared if and only if $K$ contains neither $a$ nor $b$ 
and every edge from $K$ to $\lk(a)$ or $\lk(b)$ is an edge to $\lk(a)\cap\lk(b)$.
This means that $K$ is a component of $\G-(\lk(b)\cap \lk(a))$ that does not contain $a$ or $b$.
\end{proof}

The above classification of components of $\G - \st(a)$ and $\G-\st(b)$ gives a quick way of describing when generators of $\psa$ or $\pso$ commute. We will use the commutator convention $[g,h]=ghg^{-1}h^{-1}$ throughout.

\begin{lemma}\label{l:comm}
Let $a$ and $b$ be nonadjacent vertices in $\G$. Then the commutator $[\pi_K^a,\pi_L^b]$ is nontrivial in $\aut(A_\G)$ if and only if one of the following conditions hold:
\begin{itemize}
\item $K$ and $L$ are the dominating components for the pair $(a,b)$.
\item Either $K$ or $L$ is dominating and the remaining component is shared.
\item We have $K=L$ (they are identical shared components for the pair $(a,b)$).
\end{itemize}
The image of the commutator in $\out(A_\G)$ is nontrivial if and only if one of the above cases holds and $(a,b)$ is an SIL-pair.
\end{lemma}

\begin{proof}
The statement about $\aut(A_\G)$ is shown in Lemma 4.7 of \cite{MR2888948}.
We note that the classification turns this statement into a straightforward exercise: in the cases listed above, find a vertex that the commutator does not fix, and in the other cases (some component is subordinate or the components are distinct shared components), show that every vertex is fixed.

Now we show the statement about $\out(A_\G)$.
First we suppose that we are not in one of the listed cases, or $(a,b)$ is not an SIL-pair.
If we are not in one of the three cases, then the commutator is trivial in $\out(A_\G)$ because it is trivial in $\aut(A_\G)$.
If $(a,b)$ is not an SIL-pair, then there are no shared components and the only interesting case is where $K$ and $L$ are both dominating.
 In this case, let $K^*$ be the union of the remaining (subordinate) components $A_1, \ldots A_k$ of $\G - \st(a)$. The product $\pi^a_{K^*}\pi^a_{K}$ is an inner automorphism. As $[\pi_{A_i}^a,\pi_L^b]=1$ for all $i$, the elements $\pi^b_{L}$ and $\pi^a_{K^*}$ commute. 
As $[\pi^a_K]=[\pi_{K^*}^a]^{-1}$ in $\out(A_\G)$, it follows that $[\pi^a_K]$ and $[\pi^b_L]$ commute.

We are left to show that if $(a,b)$ is an SIL-pair and the components $K$ and $L$ satisfy one of the above cases, then the commutator in question is also nontrivial in $\out(A_\G)$. Suppose that $K$ and $L$ are the dominating components. Then $[\pi_K^a,\pi_L^b](a)=[a,b]a[a,b]^{-1}$, and for any vertex $x$ in a shared component we have $[\pi_K^a,\pi_L^b](x)=x$. Suppose $[\pi_K^a,\pi_L^b]$ is an inner automorphism, conjugating all elements by some $g \in A_\G$. Then $gxg^{-1}=x$, so $g$ is in the centralizer of $x$ and the support of $g$ is a subset of $\st(x)$. It follows that the support of  $gag^{-1}$ is a subset of  $\st(x) \cup \{a\}$. This is a contradiction as $b$ is not in $st(x)$ and  \[ \supp(gag^{-1})=\supp([\pi_K^a,\pi_L^b](a))=\{a,b\}.\] Hence $[\pi_K^a,\pi_L^b]$ is nontrivial in $\out(A_\G)$. A similar argument applies in the remaining two cases. \end{proof}

In particular, one sees that $\pi^a_L$ and $\pi^b_K$ commute in $\pso$ unless $(a,b)$ is an SIL-pair.
Lemma~\ref{l:comm} makes it easy to identify the standard generators in the center of $\pso$; we leave this as an exercise to the reader.

\subsection{Presentations for $\psa$ and $\pso$}

Toinet \cite{Toinet} gave a presentation of $\psa$, and Koban--Piggott adapted Toinet's presentation as follows:

\begin{theorem}[Toinet (\cite{Toinet}, Theorem 3.1), Koban--Piggott (\cite{KP}, Theorem 3.3)] \label{t:psapresentation}
The group $\psa$ 
has a finite presentation consisting of the standard generating set and relations of the form:

\begin{enumerate}[(R1)]
\item $[\pi_K^a,\pi_L^b]=1$ when $[a,b]=1$.
\item $[\pi_K^a,\pi_L^b]=1$ when $K \cap L= \emptyset$, $b \not\in K$ and $a \not \in L$.
\item $[\pi_K^a,\pi_L^b]=1$ when $\{a\} \cup K \subset L$ or $\{b\} \cup L \subset K$.
\item \label{r4} $[\pi_{K}^a\pi_{L}^a,\pi_L^b]=1$ when $b \in K$ and $a \not\in L$. 
\end{enumerate}
\end{theorem}

Note that the case $[a,b]=1$ includes when $a=b$. In the language of Lemma~\ref{l:components}, the relation (R2) corresponds to distinct non-dominating components, and the relation (R3) corresponds to when one component is dominating and the remaining component is subordinate for the pair $(a,b)$. The repetition of $L$ in (R\ref{r4}) is not a misprint---the only time such a relation is not implied by (R1)--(R3) is when $K$ is a dominating component and $L$ is a shared component for $a$ and $b$ (in particular, $(a,b)$ forms an SIL). It follows that if $\G$ contains no SILs then $\psa$ is isomorphic to a right-angled Artin group (this was originally shown by Charney--Ruane--Stambaugh--Vijayan in \cite{CRSV}). We therefore call relations of the form (R4) \emph{SIL relations}.

As $\pso$ is obtained from $\psa$ by taking the quotient by the normal subgroup consisting of inner automorphisms, this implies:

\begin{corollary} \label{c:pres}
The group $\pso$ is finitely presented, with a presentation given by the image of the standard generating set in $\out(A_\G)$ and relations of the form:
\begin{enumerate}[(R1)]
\item $[\pi_K^a,\pi_L^b]=1$ when $[a,b]=1$
\item $[\pi_K^a,\pi_L^b]=1$ when $K \cap L= \emptyset$, $b \not\in K$ and $a \not \in L$
\item $[\pi_K^a,\pi_L^b]=1$ when $\{a\} \cup K \subset L$ or $\{b\} \cup L \subset K$.
\item $[\pi_{K}^a\pi_{L}^a,\pi_L^b]=1$ when $b \in K$ and $a \not\in L$. 
\item $\prod_{K \in I_a}  \pi_{K}^a=1$ where the product is taken over the set $I_a$ of connected components of $\G - \st(a)$.
\end{enumerate}
\end{corollary}

\subsection{The support graph}

The \emph{support graph} $\Delta_a$ gives a combinatorial description of how the roles of the components of $\G-\st(a)$ for the pair $(a,b)$ change as we vary $b$ in $\G$. We repeat the definition from the introduction:

\begin{definition} For each vertex $a \in \G$ the \emph{support graph} $\Delta_a$ has a vertex for each component of $\G - \st(a)$. There is an edge between two components $K$ and $L$ if there exists a vertex $b$ such that $K$ is the dominating component with respect to $b$ (equivalently $b \in K$) and $L$ is a shared component of both $\G - \st(a)$ and $\G -\st(b)$.
\end{definition}

In other words, each edge in $\Delta_a$ is a \emph{dominating-shared pair}: a pair of components of the form $\{[b]_a,L\}$, where $L$ is a shared component for the pair $(a,b)$. Furthermore:

\begin{lemma}[Star Lemma]
Let $(a,b)$ be an SIL-pair. There is a unique connected component $C$ of $\Delta_a$ containing the dominating component $[b]_a$ and all shared components of $\G - \st(a)$ for the pair $(a,b)$. These vertices consist of a subset of the star of $[b]_a$. If $\Delta_a$ is a forest, then every shared component $L$ is adjacent to $[b]_a$ and a (possibly empty) set of subordinate components for the pair $(a,b)$. 
\end{lemma}

\begin{proof}
From the definition of $\Delta_a$ each shared component $L$ for the pair $(a,b)$ is connected by an edge to $[b]_a$ and makes up a subset of the star of $[b]_a$ in $\Delta_a$. Hence $[b]_a$ and the shared components for $(a,b)$ lie in the same connected component of $\Delta_a$. If two shared components $L$ and $L'$ are adjacent then there exists a loop in $\Delta_a$ through $L$, $L'$ and $[b]_a$. This cannot happen if $\Delta_a$ is a forest.
\end{proof}

The support graphs let us define a large set of central elements in $\pso$. For $C$ a component of $\Delta_a$, let $\zg{a}{C}$ denote the partial conjugation
\[\zg{a}{C}=\prod_{K\in C}\pi^a_K 
\] 
\begin{proposition} \label{p:r2} Let $C$ be a component of $\Delta_a$. 
The element $\zg{a}{C}$ is central in $\pso$.
\end{proposition}

\begin{proof}
This follows from the fact that pairs of standard generators $\pi_K^a,\pi_L^b$ commute in $\out(A_\G)$ unless $(a,b)$ is an SIL-pair and $K$ and $L$ fall into one of the three cases from Lemma~\ref{l:comm}.   Fix a standard generator $\pi^b_L$ of $\pso$. If $(a,b)$ does not form an SIL-pair, then $[\pi_K^a,\pi_L^b]=1$ for all components $K$ of $\G - \st(a)$, so $\pi_L^b$ commutes with $\zg{a}{C}$. The same assertion also holds if $L$ is subordinate for the pair $(a,b)$. We may therefore assume that $(a,b)$ is an SIL-pair and $L$ is either a dominating or shared component of $\G - \st(b)$. The Star Lemma tells us that either every vertex of $C$ is a subordinate component for $(a,b)$, or $C$ contains all of the dominating and shared components.  In the first case, $\zg{a}{C}$ is a product of elements which commute with $\pi_L^b$, so $\zg{a}{C}$ commutes with $\pi_L^b$. Otherwise, as we are in $\out(A_\G)$: $${\zg{a}{C}} =\prod_{K \not\in C} ({\pi_K^a})^{-1}. $$ Each $K$ in this product is subordinate, so $\pi_L^b$ commutes with every term and also commutes with $\zg{a}{C}$. Hence $\zg{a}{C}$ commutes with every generator and is central in $\pso$.
\end{proof}

Note that $\zg{a}{C}$ is inner and trivial in $\pso$ if and only if $\Delta_a$ is connected.

\begin{remark}
When $\pso$ is a RAAG, our graphical presentation of $\pso$ will prove that elements of the form $\zg{a}{C}$ form a free (abelian) generating set of the center of $\pso$. 
It would be interesting to know whether the center is still free abelian, and whether these elements form a generating set, in the case that $\pso$ is not RAAG.
\end{remark}

\section{Subspace arrangements in vector spaces}\label{s:homology}

\subsection{A chain complex for subspace arrangements}\label{s:cc}

We fix a field $\mathbb{K}$ and work with vector spaces over $\mathbb{K}$.
A \emph{subspace arrangement} is a pair $(V,\mathcal V)$ where
$V$ is a vector space and $\mathcal{V}=(V_j)_{j\in J}$ is a collection of subspaces.
We may define a chain complex $C_*(V,\mathcal{V})$ as follows. 
We define $C_{k}$ to be trivial for $k<0$ and we define $C_0$ to be the vector space $V$.
For $k \geq 1$ we define $C_k$ by a vector space presentation.
 $C_k$ is the vector space over $\mathbb{K}$ spanned by tuples $(V_1,\ldots,V_k,v)$ such that:
\begin{itemize} 
\item $V_1,\ldots, V_k \in \mathcal{V}$,
\item $v \in V_1 \cap \ldots \cap V_k$,
\end{itemize}
subject to the relations that:
\begin{itemize}
\item $\lambda\cdot (V_1,\ldots,V_k,v)= (V_1,\ldots,V_k,\lambda v)$ for all $\lambda \in \mathbb{K}$
\item $(V_1,\ldots,V_k,v) + (V_1,\ldots,V_k,w) = (V_1,\ldots,V_k,v+w)$
\item for any permutation $\sigma$, we have $(V_1,\ldots,V_k,v)=\mathrm{sign}(\sigma)(V_{\sigma(1)},\ldots, V_{\sigma(k)},v)$;
\item if $V_i =V_j$ for some $i,j$, then $(V_1,\ldots,V_k,v)=0$ (this is implied by the above bullet point unless the field $\mathbb{K}$ is of characteristic $2$).
\end{itemize}
The boundary map $\partial_k\colon C_k \to C_{k-1}$ is defined by: $$\partial_k((V_1,\ldots,V_k,v))=\sum_{i=1}^k (-1)^{i-1} (V_1,\ldots,\hat{V}_i,\ldots,V_k,v),$$
where
$(V_1,\ldots,\hat{V}_i,\ldots,V_k,v)$ is the element of $C_{k-1}$ given by deleting the $i$th entry from the tuple. 
For $k=1$, the boundary map is defined by:
$$\partial_1((v_j)_{j \in J})= \sum_{j \in J} v_j.$$
This makes sense because $C_1=\oplus_{j\in J} V_j $ is simply the direct sum of the subspaces from $\mathcal{V}$.

\begin{remark}\label{r:repeatedsubspaces}
One can allow repetitions of subspaces in $\mathcal{V}$. In this case one must be careful to view the symmetrization given by bullet points (3) and (4) by treating vector spaces as equivalent if $V_i=V_j$ \emph{as indexed elements of $\mathcal{V}$} rather than just as subspaces of $V$. 
Adding a redundant subspace does not change the homology (see Proposition~\ref{p:max}).
We allow redundancy because it will simplify a later argument.
\end{remark}

\begin{proposition}
With the boundary maps $\partial_k$, the vector spaces  $C_k(V,\mathcal{V})$ form a well defined chain complex.
\end{proposition}

\begin{proof}
This is a straightforward exercise and we omit the details.
The most interesting part of the proof is the fact that $\partial_{k-1}\circ \partial_k=0$.
As often happens with chain complex boundary maps, this is a result of the sign convention:
for $\bar v=(V_1,\dotsc,V_k,v)\in C_k$,
the sum that we get by expanding $\partial_{k-1}\circ \partial_k(\bar v)$ contains each
$(V_1,\dotsc,\hat{V}_i,\dotsc,\hat{V}_j,\dotsc,V_k,v)$ twice, with opposite signs.
This is because $V_j$ is in the $j$th position of $\bar v$, but in the $(j-1)$st position of $(V_1,\dotsc,\hat{V}_i,\dotsc,V_k,v)$.
\end{proof}

\begin{definition}
For any subspace arrangement $(V,\mathcal{V})$ we define $H_*(V,\mathcal{V})$ to be the homology of the chain complex $C_*(V,\mathcal{V})$.
\end{definition}

As the image of  $\partial_1$ is equal to the span of $\mathcal{V}$, we have a description of $H_0$ as: \[H_0(V,\mathcal{V})\cong V / \vspan(\mathcal{V}).\]

For finite collections of subspaces, there is a more explicit description of the chain complex, which we give in the next section.

\subsection{Finite subspace arrangements}

Suppose that $\mathcal{V}=\{V_1,\ldots, V_n\}$ is a finite collection of subspaces of a vector space $V$  indexed by the set $I=\{1,\ldots,n\}$ with the natural ordering. We let $J=\{j_1,\ldots, j_k\}$ vary over all subsets of $I$ of size $k$ with $j_1 < j_2 < \cdots < j_k$ and define \[V_J=V_{j_1} \cap \ldots \cap V_{j_k}.\] The ordering of $I$ removes the need to symmetrize with respect to permuting terms in tuples, and gives a simpler description of each $C_k$ as the direct sum: \[ C_k = \bigoplus_{J\subset I, |J|=k} V_J .\]
For any $k$ and any $J\subset I$ with $|J|=k$, 
let $J_i$ be the set obtained from $J$ by removing the $i$th term, so that \[V_{J_i}=V_{j_1} \cap \ldots \cap \hat{V}_{j_i}\cap\ldots\cap V_{j_k}.\] 
Let $\partial_J^i\colon V_J \to V_{J_i}$ be the inclusion map, let $p_J : C_k \to V_J$ be the natural projection onto $V_J$ and let $\iota_{J_i} : V_{J_i} \to C_{k-1}$ be the inclusion map of $V_{J_i}$ as a factor of $C_{k-1}$.
The boundary map defined in Section~\ref{s:cc} may be rewritten as:
\[\partial_i(v) = \sum_{J}\sum_{i=1}^k (-1)^{i-1} \iota_{J_i}\circ\partial_J^i \circ p_J(v), \]
where the left hand sum ranges over all $J\subset I$ with $|J|=k$.
We define $\partial_1$ the same way as before. 
These maps are reasonably easy to write explicitly in examples. For instance, $\partial_2\colon C_2 \to C_1$ maps a vector $v\in V_i \cap V_j$ to the tuple $(0,0,\ldots,v,0,\ldots,-v,0\ldots,0)$ in $C_1=V_1\oplus V_2\cdots \oplus V_n$, where the nonzero terms occur in the $i$th and $j$th positions.

\begin{example}\label{example:h1nontrivial}
 Let $V$ be $\mathbb{K}^2$ with basis $x=(1,0)$ and $y=(0,1)$. Let $V_1$ be the $x$-axis, let $V_2$ be the $y$-axis, let $V_3$ be the subspace given by the diagonal line spanned by $x + y$, and let $\mathcal{V}=\{V_1,V_2,V_3\}$. 
Then $C_0=\mathbb{K}^2$, the space $C_1=\langle x\rangle \oplus \langle y \rangle \oplus \langle x +y \rangle$ is 3-dimensional, and 
each $C_k$ for $k \geq 2$ is trivial as no pair of distinct subspaces intersects nontrivially. 
As these subspaces span $\mathbb{K}^2$, the map $\partial_1\colon C_1\to C_0$ is surjective, and $H_0(V,\mathcal{V})=0$.
The space of $1$-cycles is $1$-dimensional, and is spanned by the cycle $(x,y,-x-y)$.
Since there are no nontrivial $1$-boundaries, this means that $H_1(V,\mathcal{V})$ is 1-dimensional, and all other homology vector spaces are trivial. 
\end{example}

\begin{example} \label{example:h2nontrivial}
Let $V$ be $\mathbb{K}^3$ with basis $x,y,z$. Let $\mathcal{V}$ be the collection of subspaces defined by \begin{align*} V_1 &= \langle y,z\rangle & V_2&=\langle x+ y,z\rangle \\ V_3&=\langle x, y+z \rangle & V_4 &= \langle x,y \rangle \end{align*} There are 6 intersections $V_i \cap V_j$ with $i < j$ given by: \begin{align*} V_1\cap V_2 &= \langle z \rangle  & V_1 \cap V_3 &= \langle y+z \rangle & V_1 \cap V_4 &=\langle y\rangle \\
V_2 \cap V_3 &=\langle x+y+z\rangle & V_2 \cap V_4 &=\langle x+y \rangle & V_3 \cap V_4 &= \langle x\rangle \end{align*}
The above calculation implies that each intersection of distinct triples in $\mathcal{V}$ is trivial, so that the chain complex is of the form \[ 0\to C_2 \xrightarrow{\partial_2} C_1 \xrightarrow{\partial_1} V\to 0\] with $\dim V =3$, $\dim C_1 = 8$ and $\dim C_2 = 6$. The map $\partial_1$ is surjective, so that $\dim (\ker \partial_1) =5$. One can check that $\partial_2$ surjects onto $\ker \partial_1$, so that $\dim (\ker \partial_2)=1$. It follows that $H_2(V,\mathcal{V})$ is 1-dimensional and the homology is trivial everywhere else.
\end{example}

\subsection{Functoriality}\label{s:funct}

Suppose that $(V,\mathcal{V})$ and $(W,\mathcal{W})$ are subspace arrangements in two vector spaces $V$ and $W$ over the same field $\mathbb{K}$. A \emph{morphism of subspace arrangements} $f\colon (V,\mathcal{V}) \to (W,\mathcal{W})$  is a linear map $f\colon V \to W$ such that for each $V' \in \mathcal{V}$, its image $f(V')$ is contained in some element of $\mathcal{W}$. In other words, for any morphism there exists a map $\alpha\colon \mathcal{V} \to \mathcal{W}$ such that $f(V') \subset \alpha(V')$ for all $V' \in \mathcal{V}$. Note that if $v \in V_1 \cap \ldots \cap V_k$ then $f(v) \in \alpha(V_1) \cap \ldots \cap  \alpha(V_k)$. Hence every choice of $\alpha$ as above gives a map \[\alpha_C \colon C_*(V,\mathcal{V}) \to C_*(W,\mathcal{W}) \] of chain complexes induced by the linear extension of the map: \[\alpha_C((V_1,\ldots,V_k,v))=(\alpha(V_1),\ldots,\alpha(V_k),f(v)). \] On $C_0$ we define $\alpha_C$ from $C_0(V,\mathcal{V})=V$ to $C_0(W,\mathcal{W})=W$ to be the linear map $f$. It is easy to check that $\alpha_C$ is a chain map, so we have an induced map on homology \[ \alpha_*\colon H_*(V,\mathcal{V}) \to H_*(W,\mathcal{W}) .\]
 Given $V' \in \mathcal{V}$, the subspace $f(V')$ may be contained in more than one element of $\mathcal{W}$, which means that the map $\alpha$ need not be unique. However, the next proposition shows that the induced map on homology depends only on $f$.

\begin{proposition}\label{p:prism}
Let $f\colon(V,\mathcal{V})\to(W,\mathcal{W})$ be a morphism of subspace arrangements. Let $\alpha,\beta\colon\mathcal{V} \to \mathcal{W}$ be maps such that $f(V') \subset \alpha(V'),\beta(V')$ for all $V' \in \mathcal{V}$. Then $\alpha_*=\beta_*$.
\end{proposition}

\begin{proof}
We will construct an explicit chain homotopy between the maps $\alpha_C$ and $\beta_C$. We use an easy modification of the \emph{prism operators} used to show that homotopic maps between two topological spaces induce the same map on homology. 
We define a degree-one map \[P:C_*(V,\mathcal{V}) \to C_{*+1}(W,\mathcal{W})\] 
that is trivial on $C_0$, and for $k\geq 1$ is defined on generators of $C_k$ by: \[ P((V_1,\ldots,V_k,v))=\sum_{i=1}^k (-1)^{i+1}(\alpha(V_1),\ldots,\alpha(V_i),\beta(V_i),\ldots,\beta(V_k), f(v)) \] and extend this map linearly. 
Following the proof in Hatcher's book \cite[Theorem~2.10]{Hatcher}, one can check that \[\partial P - P\partial = \beta_C - \alpha_C.\] Hence $P$ is a chain homotopy between $\alpha_C$ and $\beta_C$ and $\alpha_*=\beta_*$. 
\end{proof}

We have shown that any morphism of subspace arrangements induces a well-defined map on homology. A further application of the above proposition allows us to show that the homology $H_*(V,\mathcal{V})$ only depends on the maximal subspaces in $\mathcal{V}$.

\begin{proposition}\label{p:max}
Suppose that $\mathcal{V}$ is a subspace arrangement in $V$ such that each element of $\mathcal{V}$ is contained in a maximal element of $\mathcal{V}$ (this is true if $V$ is finite dimensional). Let $\mathcal{V}'$ be the family of maximal elements of $\mathcal{V}$. Then $H_*(V,\mathcal{V}) \cong H_*(V,\mathcal{V}')$.
\end{proposition}

\begin{proof}
Let $\alpha\colon\mathcal{V'} \to \mathcal{V}$ be the natural injection of $\mathcal{V'}$ into $\mathcal{V}$. We may also choose a map $\beta\colon\mathcal{V} \to \mathcal{V'}$ by picking a maximal subspace $\beta(W)$ containing each element $W \in \mathcal{V}$. Let $\alpha_C$ and $\beta_C$ be the induced maps on chain complexes with respect to the identity map from $f\colon V \to V$ to itself. Note that $\beta \circ \alpha$ is the identity map on $\mathcal{V'}$, and it follows that $\beta_C \circ \alpha_C$ is the identity map on $C_*(V,\mathcal{V}')$. Hence $\beta_* \circ \alpha_*$ induces the identity map on $H_*(V,\mathcal{V}')$. In the other direction, $\alpha_C \circ \beta_C$ is the map from $C_*(V,\mathcal{V})$ to itself induced by the map $\alpha\beta\colon \mathcal{V} \to \mathcal{V}$. 
We may apply Proposition~\ref{p:prism} to the identity morphism $f=\mathrm{id}_V\colon (V,\mathcal{V}) \to (V,\mathcal{V})$.
Here we take $\alpha'\colon \mathcal{V}\to\mathcal{V}$ to be the identity map on the family $\mathcal V$ and take map $\beta'=\alpha\beta \colon \mathcal{V} \to \mathcal{V}$; the Proposition implies $\alpha_*'=\beta_*'$.  
As $\alpha'_*$ is the identity map, so is $\beta'_*=\alpha_*\circ \beta_*$. It follows that $\alpha_*$ and $\beta_*$ are isomorphisms.
\end{proof}

\begin{corollary}\label{c:trivial} 
If $V \in \mathcal{V}$ then $H_*(V,\mathcal{V})$ is trivial.
\end{corollary}

\begin{proof}
This follows from the above as $H_*(V,\{V\})$ is trivial and isomorphic to $H_*(V,\mathcal{V})$.
\end{proof}

\begin{remark}
It is possible to characterize $H_*(V,\mathcal{V})$ as a derived functor.
We do not use this in this paper, but we outline it in this remark.

Consider a category $\mathcal{C}$ of subspace arrangements with a fixed index set $J$, whose morphisms are linear maps that send the $j$th subspace into the $j$th subspace for each $j\in J$ (this is much more restrictive than the definition we use above).
This category $\mathcal{C}$ is an additive category, but not an abelian category because epimorphisms and monomorphisms are not necessarily normal.
We consider the category $\mathcal{D}$ of cubical diagrams of vector spaces;
this is the functor category from the opposite category of the category of subsets of $J$ (with inclusions) to the category of vector spaces over $\mathbb{K}$.
It turns out that $\mathcal{C}$ embeds in $\mathcal{D}$ by sending an arrangement to the diagram of inclusions of intersections of subspaces in the arrangement.

It follows from standard arguments that $\mathcal{D}$ is an abelian category, and it is possible to show that every object is a quotient of a projective object.
The $H_0$ functor we define above corresponds to a functor $H_0$ from $\mathcal{D}$ to vector spaces.
Specifically, if an object $(V,f)$ of $\mathcal{D}$ is given by $\{V_S\}_{S\subset J}$ and $\{f_{S,T}\colon V_S\to V_T\}_{T\subset S\subset J}$, then
\[H_0((V,f))=V_{\varnothing}/\vspan(\{f_{\{j\},\varnothing}(V_{\{j\}})\}_{j\in J}).\]
This functor turns out to be right-exact.
Our homology theory functors are then the left-derived functors of the functor $H_0$.
\end{remark}

\subsection{Inclusion-exclusion}
Our next statement has a connection to the inclusion-exclusion principle, which we explain in the following remark.
\begin{remark}
Recall that the inclusion-exclusion principle allows us to count a finite union of sets $\{S_j\}_{j\in J}$ by taking an alternating sum of the counts of the intersections of these sets:
\[\Big|\bigcup_{j\in J} S_j\Big|=\sum_{k=1}^{|J|}(-1)^{k+1}\left(\sum_{I\subset J,|I|=k} \Big|\bigcap_{j\in I}S_j\Big|\right).\]
One might hypothesize an analogous statement for vector spaces, asserting that the dimension of a span of vector subspaces $\{V_j\}_{j\in J}$ is an alternating sum of the dimensions of the intersections:
\[\dim\vspan(\{V_j\}_{j\in J})\stackrel{?}{=}\sum_{k=1}^{|J|}(-1)^{k+1}\left(\sum_{I\subset J,|I|=k} \dim\Big(\bigcap_{j\in I}V_j\Big)\right).\]

This is famously false, although it holds in many simple examples.
It fails in different ways in Examples~\ref{example:h1nontrivial} and~\ref{example:h2nontrivial}, by overcounting in the first one and undercounting in the second one.

Suppose $(V,\mathcal{V})$ is an arrangement where $V$ is finite-dimensional and the subspaces in $\mathcal{V}$ span $V$ (in other words, $H_0(V,\mathcal{V})=0$).
For such an arrangement, the validity of the ``inclusion-exclusion principle for vector spaces" is equivalent to the vanishing of the Euler characteristic of $H_*(V,\mathcal{V})$.
We do not use this fact, but we leave it as an exercise for the interested reader.
\end{remark}

We do require one result that is related to inclusion-exclusion.
We are interested in the case where all subspaces in our collection $\mathcal{V}$ are generated by subsets of a fixed basis for $V$.
(We will see below that this is true for BNS invariants of RAAGs.)
In this case, if $\dim(V)$ is finite, then inclusion-exclusion clearly holds.
This means that the alternating sum in the remark above is $\dim(\vspan(\mathcal{V}))$, and that the Euler characteristic of $H_*(V,\mathcal{V})$ is 
\[\dim V -\dim \vspan(\mathcal{V})=\dim H_0(V,\mathcal{V}).\]
In fact, more is true: in this special case, the homology is trivial, except possibly for $H_0(V,\mathcal{V})$.
Our proposition refines a lemma of Koban--Piggott~\cite{KP}, which uses an inclusion-exclusion sum involving the BNS invariant of a RAAG to count the number of non-central vertices in the defining graph.
We state and prove our proposition assuming that $V$ is finite-dimensional, although this can be easily extended to the general case.

\begin{proposition}\label{p:trivial}
Let $V$ be a vector space with basis $S=\{s_1,\ldots,s_n\}$ and let $\mathcal{V}$ be a collection of subspaces of $V$ such that each $V'\in\mathcal{V}$ is spanned by a subset of $S$. Then $H_n(V,\mathcal{V})=0$ for all $n \geq 1$.
\end{proposition}

\begin{proof}
We use induction on the dimension of $V$. 
When $\dim(V)=1$, either all spaces in the collection $\mathcal{V}$ are trivial, or $V \in\mathcal{V}$. 
The result then follows from Corollary~\ref{c:trivial}. 
Now suppose the result holds for all such arrangements in vector spaces of dimension $n-1$. 
Let $V$ and $\mathcal{V}=\{V_i\}_{i\in I}$ be as in the statement of the theorem with basis $S=\{s_1,\ldots,s_n\}$. 
Let $P$ be the subspace spanned by $S'=\{s_1,\ldots,s_{n-1}\}$ and let $p\colon V\to P$ be the projection given by \[p\Big(\sum_{i=1}^n \lambda_is_i\Big)=\sum_{i=1}^{n-1}\lambda_is_i.\] 
Let \[\mathcal{P}=\{P_i=p(V_i):i \in I\}\] be the projected subspace arrangement in $P$. 
Let $Q = \langle s_n \rangle$ and let \[\mathcal{Q}=\{Q_i=Q\cap V_i: i \in I\}\] be the induced subspace arrangement in $Q$. Note that for both $(P,\mathcal{P})$ and $(Q,\mathcal{Q})$ we allow for repetitions of subspaces as described in Remark~\ref{r:repeatedsubspaces}.

Let $\alpha_C\colon C_*(V,\mathcal{V}) \to C_*(P,\mathcal{P})$ be the induced map on chain complexes coming the from projection $p\colon V \to P$ and the map $\alpha\colon \mathcal{V} \to \mathcal{P}$ given by $\alpha(V_i)=P_i$. 
The element $(V_1,\ldots,V_k,v)$ is mapped to $(P_1,\ldots,P_k,p(v))$ under $\alpha_C$. 
If $w \in P_1 \cap \cdots \cap P_k$, there exists $v \in V_1 \cap \cdots \cap V_k$ with $p(v)=w$. It follows that $\alpha_C$ is surjective. The kernel chain complex of $\alpha_C$ is spanned in $C_k(V,\mathcal{V})$ by elements of the form $(V_1,\ldots,V_k,\lambda s_n)$, and in $C_0(V,\mathcal{V})=V$ the kernel is the subspace $Q=\langle s_n \rangle$. This kernel chain complex is naturally isomorphic to $C_*(Q,\mathcal{Q})$. We then have a short exact sequence of chain complexes 
\[ 0 \to C_*(Q,\mathcal{Q}) \to C_*(V,\mathcal{V}) \stackrel{\alpha_C}{\longrightarrow} C_*(P,\mathcal{P}) \to 0 \]
which induces the long exact sequence in homology \[ \cdots \to H_{k}(Q,\mathcal{Q}) \to H_k(V,\mathcal{V}) \to H_k(
P,\mathcal{P}) \to \cdots \]
As each vector space in in $\mathcal{V}$ is spanned by a subset of $S$, each element of $\mathcal{P}$ is spanned by a subset of $S'$. Hence both $(P,\mathcal{P})$ and $(Q,\mathcal{Q})$ are subspace arrangements where each subspace is spanned by a fixed subset of some basis. For $k \geq1$, the space $H_k(P,\mathcal{P})$ is trivial by the inductive hypothesis and $H_k(Q,\mathcal{Q})$ is trivial by the dimension 1 case. This implies that $H_k(V,\mathcal{V})$ is trivial for $k \geq 1$ also.
\end{proof}

\section{BNS invariants and subspace arrangements} \label{s:bns}

\subsection{BNS invariants}
\label{ss:bns}
The Bieri--Neumann--Strebel invariant is a subset $\Sigma$ of the character sphere of a finitely generated group $G$. The character sphere $S$ of $G$ is the set \[(\hom(G;\R) \setminus \{0\}) / \sim\]
where characters are identified if they lie in the same ray in $\hom(G;\R)$: $\chi_1\sim\chi_2$ if and only if there is $\lambda>0$ with $\chi_1=\lambda\chi_2$.
The original definition of the BNS invariant from~\cite{BNS} states that $[\chi]\in S$ is in $\Sigma$ if and only if $[G,G]$ is finitely generated over a finitely generated submonoid of $\chi^{-1}([0,\infty))$.
Bieri--Neumann--Strebel also give a convenient characterization in terms of a generating set in Proposition~2.3 of~\cite{BNS}: 
 $[\chi]\in S$ is in $\Sigma$ if and only if the preimage under $\chi$ of the closed half-line $[0,\infty)$ in the Cayley graph of $G$ is connected.
We do not use the original definition or the equivalent one from the original paper; instead we prefer another equivalent definition due to Brown that we state below.

\begin{remark}
Sometimes $\Sigma$ is viewed as the first invariant in a collection $\Sigma=\Sigma^1 \supset \Sigma^2 \supset \Sigma^3 \supset \cdots$ (see \cite{BNS}). We will not be considering these higher invariants in this paper.
\end{remark}

Recall that an $\R$-tree is a geodesic metric space in which a unique arc connects any two points.
An action of $G$ on an $\R$-tree $T$ is \emph{abelian} if there exists a character $\chi$ such that $|\chi(g)|=\|g\|_T$ for all $g \in G$, where $\|g\|_T$ is the translation length of $g$ as an isometry of $T$. We say that $T$ \emph{realizes} $\chi$. Note that for each $\chi$ there is a natural abelian action of $G$ on a line realizing $\chi$. Any abelian action realizing a nontrivial character fixes one or two points in the boundary $\partial T$ of $T$. When there is a unique fixed point in $\partial T$ we say that the action is \emph{exceptional}. 

Let $T$ be an exceptional action realizing a character $\chi$ with fixed end $e \in \partial T$. Let $(g_n)$ be a sequence of elements of $G$ such that for some (equivalently, any) point $x \in T$ the orbit $g_n \cdot x$ converges to $e$. The sequence $(\chi(g_n))$ converges to either $+\infty$ or $-\infty$. We say that \emph{the invariant end is at $+\infty$} in the former case, and $-\infty$ in the latter. This is independent of any choices made above. Swapping $\chi$ with $-\chi$ will then swap the location of the invariant end. The following definition of $\Sigma$ is due to Brown \cite{Brown}, who showed that it is equivalent to the original definition from \cite{BNS}. 

\begin{definition}
An element $[\chi] \in S$ is in $\Sigma$ if there exists no exceptional action of $G$ on an $\R$-tree $T$ realizing $\chi$ with the invariant end at $-\infty$.
\end{definition}

Note that Brown's definition allows one to consider $\Sigma$ even in the case that $G$ is not finitely generated. 

Rather than considering the BNS invariant as a subset of the character sphere, for most of the paper we will consider the preimage of $\Sigma$ in $\hom(G;\R)$. Let \[p\colon (\hom(G;\R)-\{0\}) \to S\] be the quotient map to the character sphere.
We say that $\chi \in \hom(G;\R)$ \emph{lies in the complement of the BNS invariant} if $\chi \not \in p^{-1} (\Sigma)$. 
The complement of the BNS invariant may then be viewed as a subspace arrangement in $\hom(G;\R)$.

\begin{definition}
Let $G$ be group.
We define $\mathcal{V}_G$ to be the set of maximal subspaces in $\hom(G;\R)$ contained in the complement of the BNS invariant.
We define $H_*(\mathcal{V}_G)$ to be the subspace arrangement homology $H_*(V,\mathcal{V}_G)$, where our ambient space $V$ is always $\hom(G;\R)$.
\end{definition} 

More generally, we can consider the collection of all subspaces of $\hom(G;\R)$ in $p^{-1}(\Sigma^c)\cup\{0\}$; Proposition~\ref{p:max} shows that this gives the same homology spaces as the collection of maximal subspaces $\mathcal{V}_G$.

\begin{remark}
Recall that the BNS invariant $\Sigma$ of a group $G$ is \emph{symmetric} if $\Sigma=-\Sigma$, meaning that it is invariant under the antipodal map. In this case, each character $\chi$ with $[\chi] \in \Sigma^c$ determines an entire line in $p^{-1}(\Sigma^c)\cup\{0\}$. As $\chi$ is contained in some subspace of $p^{-1}(\Sigma^c)\cup\{0\}$, it is also contained in a maximal one. Hence $p^{-1}(\Sigma^c)\cup\{0\}$ is exactly the union of the elements of $\mathcal{V}_G$. Conversely, if $\Sigma$ is not symmetric then $\cup \mathcal{V}_G$ is a proper subset of $p^{-1}(\Sigma^c)\cup\{0\}$. 
Even if $\Sigma$ is symmetric and $\hom(G;\R)$ is finite dimensional, as far as we know it is still possible for $\mathcal{V}_G$ to be an infinite family.  
\end{remark}

\begin{remark}
One can instead take the larger family $\mathcal{V}^+_G$ spanned by characters $\chi$ which are realized by some exceptional action on an $\mathbb{R}$-tree (in other words, either $\chi$ or $-\chi$ lies in $p^{-1}(\Sigma^c)$). One can view $\mathcal{V}_G$ as the arrangement obtained by removing characters corresponding to $\Sigma \cup -\Sigma$ from $V$, whereas for $\mathcal{V}_G^+$ one only removes characters corresponding to elements of $\Sigma\cap-\Sigma$. When $\Sigma$ is non-symmetric there are examples where $H_*(\mathcal{V}_G,V)$ and $H_*(\mathcal{V}^+_G,V)$ are different (such examples can be found in \cite{BNS, Brown}).
\end{remark}

\subsubsection{Maps between groups}

When $f\colon G\to H$ is a surjective homomorphism, an exceptional abelian action of $H$ on a tree induces an exceptional abelian action of $G$. This does not change the location of the invariant end with respect to the characters $\chi:H \to \mathbb{R}$ and $f^*(\chi):G \to \mathbb{R}$. Hence we have the following well-known fact:

\begin{proposition}\label{p:char}
Let $f\colon G \to H$ be a surjective map and \[f^*\colon \hom(H;\R) \to \hom(G;\R)\] the induced map on character spaces. If $\chi \in \hom(H;\R)$ is in the complement of the BNS invariant of $H$, then $f^*(\chi)=\chi \circ f$ is in the complement of the BNS invariant of $G$.
\end{proposition} 

It follows that if $f\colon G\to H$ is surjective, 
then $f$ induces a morphism of subspace arrangements\[f^*\colon (\hom(H;\R),\mathcal{V}_H) \to (\hom(G;\R),\mathcal{V}_G). \] 
This in turn gives a map $(f^*)_*\colon H_*(\mathcal{V}_H) \to H_*(\mathcal{V}_G)$ on homology as described in Section~\ref{s:funct}, although we will not need this in the work that follows.

To summarize, we have defined
\[G\mapsto H_*(\mathcal{V}_G),\]
a contravariant functor from the category of groups with \emph{surjective} homomorphisms to the category of graded vector spaces over $\R$.
Such a thing superficially resembles a cohomology theory of groups.
It would be interesting to characterize this invariant in terms of cohomology.

\subsection{Right-angled Artin groups}

Suppose that $G$ is a right-angled Artin group $A_\G$. For a vertex $a$ of $\G$, let $\chi_a \colon A_\G \to \R$ be the character defined on generators by \[ \chi_a(v) = \begin{cases} 1 &\text{if $v=a$}\\ 0& \text{if $v \neq a$}\end{cases} \] The abelianization of $A_\G$ is a free abelian group generated by the images of the vertices in $H_1(A_\G;\Z)$ and the characters $\chi_a$ define a basis of $\hom(G;\R)$.

For any character $\chi \in \hom(A_\G;\R)$, we define the \emph{support} $\supp(\chi)$ to be the full subgraph of $\G$ spanned by the vertices $v$ such that $\chi(v) \neq 0$. The support is \emph{dominating} if every vertex in $\G$ is either contained in, or adjacent to, a vertex in $\supp(\chi)$. 

\begin{theorem}[Meier--VanWyk, \cite{MVW}]\label{t:MVW}
Let $\chi \in \hom(A_\G;\R) -\{0\}$. Then $[\chi] \in \Sigma(A_\G)$ if and only if $\supp(\chi)$ is connected and dominating.
\end{theorem}

\begin{proposition}
The set $\mathcal{V}_G$ of maximal subspaces in the complement of the BNS invariant is the set of vector spaces of the form \[V_S = \langle \chi_a : a\in S \rangle \]
for each maximal subset $S$ of vertices in $\G$ spanning a disconnected subgraph of $\G$.
\end{proposition}

\begin{proof}
If $\G'$ is a subgraph of $\G$ which is not dominating, then there is a vertex $v'$ which is not adjacent to $\G'$, so that $\G' \cup v'$ spans a disconnected subgraph of $\G$. Hence every subgraph which is either disconnected or non-dominating is contained in a maximal disconnected subgraph of $\G$. Combining this with Theorem~\ref{t:MVW}, the support of every character that lies in the complement of the BNS invariant of $A_\G$ is contained in a maximal disconnected subgraph of $\G$, hence lies in $V_S$ for some $S$ as above.
\end{proof}

\begin{corollary}
Let $G=A_\G$ be a right-angled Artin group. Then:
\[ \dim(H_n(\mathcal{V}_G)) =\begin{cases} \mathrm{rank}(Z(A_\G)) &\text{if $n=0$} \\ 0 &\text{if $n >0$}\end{cases} \]
\end{corollary}

\begin{proof}
Each element of $\mathcal{V}_G$ is spanned by a subset of our basis for $\hom(A_\G;\R)$. For $n \geq1$, each $H_n(\mathcal{V}_G)$ is trivial by Proposition~\ref{p:trivial}. We are then left to find \[H_0(\mathcal{V}_G)=\hom(G;\R) / \vspan(\mathcal{V}_G).\]A vertex $a \in \G$ lies in a disconnected full subgraph of $\G$ if and only if $\st(a)$ if not equal to the whole of $\G$. In other words, $\chi_a$ is contained in some element of $\mathcal{V}_G$ unless $a$ is central in $A_\G$. It follows that $\dim(H_0(\mathcal{V}_G))$ is equal to the rank of the center of $A_\G$.
\end{proof}

In particular, the Euler characteristic of $H_*(\mathcal{V}_G)$ is equal to the rank of the center of $A_\G$ and is non-negative (cf. Theorem~\ref{t:KP2}).

\subsection{Pure symmetric automorphisms}

Now suppose that $G=\psa$ and let $X$ be the standard generating set of $G$. 
For $a\in\Gamma$ and $K\in\Delta_b$, we let $\chi^a_K$ be the character defined on generators by \[ \chi^a_K(\pi^b_L) = \begin{cases} 1 &\text{if $\pi^a_K=\pi^b_L$}\\ 0& \text{otherwise}\end{cases} \]
It follows from
Toinet's presentation 
that the abelianization of $\psa$ is a free abelian group, and the standard generators map bijectively to a free generating set.
This means that each $\chi^a_K$ is a well-defined element of $\hom(\psa;\R)$
and the elements $\chi_x$ form a basis of $\hom(\psa;\R)$. 
As before, we may define the support $\supp(\chi)$ of a character $\chi$ to be the subset of the standard generating set $X$ consisting of all generators such that $\chi(\pi_K^a) \neq 0$. 
Koban and Piggott characterize elements of $\Sigma(\psa)$ according to their support in a similar fashion to Meier--VanWyk. 
They first define the following nice subsets of the standard generating set:

\begin{definition} \label{d:psets}
A subset $S \subset X$ is a \emph{p-set} if
\begin{itemize}
\item for each vertex $a$ of $\Gamma$, there is at most one partial conjugation in $S$ with multiplier $a$, and
\item $S$ has a nontrivial partition $S=S_1\cup S_2$ such that for every $\pi^a_K\in S_1$ and $\pi^b_L\in S_2$, we have $a\in L$ and $b\in K$ ($L$ and $K$ are the dominating components for the pair $(a,b)$).
\end{itemize}

A subset $S \subset X$ is a \emph{$\delta$-p-set} if
\begin{itemize}
\item 
for each vertex $a$ of $\Gamma$, there are exactly two or zero partial conjugations in $S$ with multiplier $a$, and
\item 
$S$ has a nontrivial partition $S=S_1\cup S_2$ such that for every $\pi^a_K\in S_1$ and $\pi^b_L$ in $S_2$, we have $a\in L$ or $b\in K$ or $K=L$ (so $L$ is the dominating component $[a]_b$ or $K$ is the dominating component $[b]_a$ or $K$ and $L$ are the same shared component).
\end{itemize}
\end{definition}

The p-sets here give exceptional characters similar to those occurring in RAAGs, whereas the $\delta$-p-sets only appear when $\G$ has an SIL-pair. The complement of the BNS invariant of $\psa$ can be characterized as follows:

\begin{theorem}[Koban--Piggott, \cite{KP}] \label{t:kp}
Let $\chi\colon \psa \to \R$ be nonzero character. Then $\chi$ is in the complement of the BNS invariant if and only if
\begin{itemize}
\item  $\chi$ is nontrivial on some inner automorphism and the support of $\chi$ is a subset of a p-set, or
\item  $\chi$ is trivial on every inner automorphism and the support of $\chi$ is a subset of a $\delta$-p-set.
\end{itemize}
\end{theorem}

In the second case of the above theorem, as $\chi$ is trivial on every inner automorphism, it follows that  $\chi(\pi^a_K)=-\chi(\pi^a_L)$ for each pair of elements $\pi^a_K, \pi^a_L$ with the same multiplier in its associated $\delta$-p-set. This gives enough information to describe $\mathcal{V}_G$.

\begin{proposition}\label{p:psabns} Let $G=\psa$ and let $\mathcal{V}_G$ be the set of maximal subspaces in the complement of the BNS invariant. For each maximal p-set $S\subset X$ there is a subspace $V_S \in \mathcal{V}_G$ given by \[ V_S= \langle \{\chi^a_K : \pi^a_K\in S \}\rangle \]
and for each maximal $\delta$-p-set $S \subset X$ there is a subspace $V_S \in \mathcal{V}_G$ of the form \[V_S = \langle\{ \chi^a_K- \chi^a_L :\pi^a_K, \pi^a_L\in S \}\rangle \] 
Furthermore, each element of $\mathcal{V}_G$ is one of these two types.
\end{proposition}

Koban and Piggott used this description to take an alternating sum of dimensions of intersections of spheres in $\Sigma^c(G)$. 
Intersections of spheres in $\Sigma^c(G)$ correspond to intersection of subspaces in $\mathcal{V}_G$.
Using our terminology, we rephrase their result as follows:

\begin{theorem}[Koban--Pigott~\cite{KP}]\label{t:KP2}
Let $G=\psa$. If $\G$ contains no separating intersection of links then the Euler characteristic of $H_*(\mathcal{V}_G)$ is zero. Otherwise, the Euler characteristic of $H_*(\mathcal{V}_G)$ is strictly negative.
\end{theorem}
 
 \subsection{Pure symmetric outer automorphisms}
 We now turn our attention to $\pso$. Let \[ f\colon\psa \to \pso \] be the quotient map, and let \[ f^* \colon \hom(\pso;\R) \to \hom(\psa;\R) \] be the dual map on characters given by $f^*(\chi)=\chi \circ f$. As $f$ is surjective, the map $f^*$ is injective, with image given by the characters $\chi \in \hom(\psa;\R)$ that are trivial on the inner automorphisms. In other words, if $\Delta_a$ is the support graph for some vertex $a \in \G$, we have \[\sum_{K \in \Delta_a} \chi(\pi_K^a)=0.\]
We identify $\hom(\pso;\R)$ with this subspace of $\hom(\psa;\R)$. 
This allows us to talk about the support of a character on $\pso$; it is the support of of the character on $\psa$ we get by composing with the projection $f$.

To proceed, we need the following well-known fact, which is stated in~\cite{KP}.  We do not give a proof here.
\begin{lemma}[\cite{KP}]\label{l:freeprodchar}
If $\chi\colon G\to \R$ is a nontrivial character on a group $G$ that factors through a surjective map $G\to A*B$, where $A*B$ is a nontrivial free product, then $[\chi]$ is in the complement of the BNS invariant of $G$.
\end{lemma}

 \begin{proposition} \label{p:bnspso}
Let $A_\Gamma$ be a RAAG and let $\chi\colon \pso\to \R$ be nonzero.
The class $[\chi]$ is not in the BNS invariant if and only if the support of $\chi$ is a subset of a $\delta$-p-set.
\end{proposition}

\begin{proof}
By Proposition~\ref{p:char}, if $\chi$ is in the complement of the BNS invariant of $\pso$, then $f^*(\chi)$ is in the complement of the BNS invariant of $\psa$. 
Since $f^*(\chi)$ is in the image of $f^*$, it is trivial on every inner automorphism.
Then by Theorem~\ref{t:kp} it has a support which is a subset of a $\delta$-p-set. 

Conversely, given any character $\chi$ whose support is a $\delta$-p-set, we need to show that $[\chi]$ is not in $\Sigma$. 
Following~\cite{KP}, we find a surjection $\phi\colon \pso \to A_1*A_2$ to a nontrivial free product which $\chi$ factors through.
By Lemma~\ref{l:freeprodchar}, it will follow that $[\chi]\in\Sigma^c$. 
Let $S$ be the $\delta$-p-set which is the support of $\chi$, and let $S_1 \cup S_2$ be a partition of $S$ given in Definition~\ref{d:psets}. 
Each multiplier $a$ that appears in $S$ has two elements $\pi^a_{K_1}$ and $\pi^a_{K_2}$, both of which lie on one side of the partition $S_1 \cup S_2$. Furthermore, $\chi(\pi^a_{K_1})=-\chi(\pi^a_{K_2})$. Let $A_1$ be the free abelian group on the multipliers that appear in $S_1$ and $A_2$ the free abelian group on the set of multipliers that appear in $S_2$. We map $\pso$ to $A_1*A_2$ by sending $[\pi^a_{K_1}]$ to $a$, sending $[\pi^a_{K_2}]$ to $-a$, and every other generator with multiplier $a$ to the trivial element. If $b$ is a multiplier that occurs on the other side of the partition with corresponding elements $\pi^b_{L_1}$ and $\pi^b_{L_2}$, then each commutator $[\pi^a_{K_i},\pi^b_{L_j}]$ is nontrivial in $\pso$. Furthermore, one can check that the map to $A_1 * A_2$ respects all relations in the presentation of $\pso$ and is therefore well-defined. Hence $[\chi] \in \Sigma^c$.
\end{proof}

\begin{corollary} \label{co:psobnsarrangement}
Let $G=\pso$ and let $\mathcal{V}_G$ be the set of maximal subspaces in the complement of the BNS invariant. 
The family $\mathcal{V}_G$ consists exactly of the subspaces of the form 
 \[V_S = \langle \{\chi^a_K- \chi^a_L :\pi^a_K, \pi^a_L\in S \}\rangle \] 
for each maximal $\delta$-p-set $S$.
\end{corollary}

Our next goal is to show that $H_1(\mathcal{V}_G)$ is nontrivial for $G=\pso$ under certain conditions.
To do this, we build a cycle and show that it represents a nontrivial homology class.
As is often the case with homology theories, it is convenient to do this by pairing our cycle with a cocycle.

We do not give a full treatment of a cohomology theory of subspace arrangements here.
However, we make the following  definition: for a subspace arrangement $(V,\mathcal{V})$ over $\mathbb{K}$, we define 
\[C^*(V,\mathcal{V})=\hom(C_*(V,\mathcal{V}),\mathbb{K}),\]
and for $f\in C^n(V,\mathcal{V})$, define $df=f\circ\partial$.
This is a cochain complex and we define cocycles, coboundaries and cohomology as usual.

This means that a $1$-cochain $f$ in $C^1(V,\mathcal{V})$ is determined by a family  $\{f_W\}_{W \in \mathcal{V}}$ of linear functionals on each subspace; each $f_W$ is the restriction of $f$ to the $W$-summand of $\bigoplus\mathcal{V}=C_1(V,\mathcal{V})$.
Such a collection of functionals determines a cocycle if, for any two subspaces $W_1$ and $W_2$, the linear maps $f_{W_1}$ and $f_{W_2}$ agree on $W_1 \cap W_2$ (this is easily seen to be equivalent to $f\circ \partial=0$).
The cocycle $f$ represents the trivial cohomology class if and only if there exists a linear functional $\tilde f\colon V \to \R$ such that each $f_W$ is the restriction of $\tilde f$ to $W$ (this is the same as saying that $f=\tilde f\circ \partial$).

Suppose $c=(c_W)_{W\in\mathcal{V}}$ is a $1$-chain in $C_1(V,\mathcal{V})=\bigoplus\mathcal{V}$.
If the $1$-cocycle $f$ is expressed as a family of functionals $\{f_W\}_{W\in\mathcal{V}}$, then $f(c)$ is the sum $\sum_{W\in\mathcal{V}}f_W(c_W)$.
As usual, the evaluation of $1$-cocycles on $1$-cycles descends to a well defined evaluation of cohomology classes on homology classes.
In particular, if $c$ is a $1$-boundary, then $f(c)=0$ for any $1$-cocycle $f$.
So if $f(c)\neq 0$ for some cocycle, then $c$ represents a nontrivial homology class.

\begin{proposition}
Suppose $A_\Gamma$ is a RAAG such that for some vertex $a\in  \Gamma$, the support graph $\Delta_a$ contains a loop.
Let $G=\pso$ and let $\mathcal{V}_G$ be the excluded subspace configuration for the BNS invariant of $G$ in $V=\hom(G;\R)$.
Then $H_1(\mathcal{V}_G)\neq0$.
\end{proposition}

\begin{proof}
Let $(K_1,\dotsc,K_n)$ be a loop in $\Delta_a$ involving $n\geq 3$ distinct vertices.
By the definition of $\Delta_a$, for each $i$ (from $1$ to $n$ and counting modulo $n$), there is an element $b_i$ such that either (1) $b_i\in K_i$, and $K_{i+1}$ is a shared component of $b_i$ and $a$; or (2) $b_i\in K_{i+1}$, and $K_i$ is a shared component of $b_i$ and $a$.
This implies that either $\{\pi^a_{K_i},\pi^a_{K_{i+1}}\}\cup\{\pi^{b_i}_{[a]},\pi^{b_i}_{K_{i+1}}\}$ or $\{\pi^a_{K_i},\pi^a_{K_{i+1}}\}\cup\{\pi^{b_i}_{[a]},\pi^{b_i}_{K_{i}}\}$ is a $\delta$-p-set.
Each of these sets is contained in a maximal $\delta$-p-set.
So for $i=1,\dotsc,n$, let $S_i$ be a maximal $\delta$-p-set with $\pi^a_{K_i},\pi^a_{K_{i+1}}\in S_i$, and for $i=n+1,\dotsc, m$, let $\{S_i\}_i$ label the remaining maximal $\delta$-p-sets in any order. 
Let $V_i$ be the span of $S_i$ for $i=1,\dotsc,m$; then Corollary~\ref{co:psobnsarrangement} says that $\mathcal{V}_G=\{V_1,V_2,\ldots,V_m\}$.
We build the following element of $C_1(\mathcal{V}_G)=\bigoplus_{i=1}^m V_i$:
\[x=(\chi^a_{K_1}-\chi^a_{K_2},\chi^a_{K_2}-\chi^a_{K_3},\dotsc,\chi^a_{K_{n-1}}-\chi^a_{K_n},\chi^a_{K_n}-\chi^a_{K_1},0,\dotsc,0)\in C_1(\mathcal{V}_G).\]
This $x$ a cycle for $H_1(\mathcal{V}_G)$, since the sum of its components is zero.

To show that $x$ represents a nontrivial homology class, we build a cocycle.
Define a set $T\subset \{1,\dotsc,m\}$ by
\[T=\big\{i\in\{1,\dotsc,m\} : \{\pi^a_{K_1},\pi^a_{K_2}\}\subset S_i\big\}.\]
We define functionals $f_{V_i}\colon V_i\to \R$ for $i=1,\dotsc,m$ 
as follows:
if $i\in T$, then $f_{V_i}(\chi)= \chi(\pi^a_{K_1})$ for $\chi\in V_i$; 
if $i\notin T$, then $f_{V_i}=0$.

To show that these functionals patch together to form a cocycle $f$,
we need to show that they agree on the intersections of their domains.
Let $1\leq i< j\leq m$.
If both $i$ and $j$ are in $T$, or both $i$ and $j$ are not in $T$, then clearly
$f_{V_i}|_{V_i\cap V_j}=f_{V_j}|_{V_i\cap V_j}$.
So suppose that $i\in T$ and $j\notin T$.
Let $\chi\in V_i\cap V_j$.
By Proposition~\ref{p:char} and Theorem~\ref{t:kp}, $\supp(\chi)$ contains exactly zero or two standard generators with multiplier $a$.
Since $\chi\in V_j$, we know $\supp(\chi)\subset S_j$, so $\supp(\chi)$ does not contain both $\pi^a_{K_1}$ and $\pi^a_{K_2}$.
But $\chi\in V_i$ and $i \in T$, so if $\supp(\chi)\subset S_i$ contains a standard generator with multiplier $a$ it contains both $\pi^a_{K_1}$ and $\pi^a_{K_2}$ with $\chi(\pi^a_{K_1})=-\chi(\pi^a_{K_2})$. 
Therefore $\supp(\chi)$ does not contain any generators with multiplier $a$.
This means that $f_{V_i}(\chi)=\chi(\pi^a_{K_1})=0=f_{V_j}(\chi)$.
The case where $j\in T$ and $i\notin T$ is identical, so we have that the $\{f_{V_i}\}_i$ agree on all pairwise intersections of spaces from $\mathcal{V}_G$.
This means that these functionals  patch together to form a cocycle $f$ in $C^1(\mathcal{V}_G)$.

Now it is enough to show that $f(x)\neq 0$.
By our numbering of $S_1,\dotsc, S_m$, we know that $T\cap\{1,\dotsc,n\}=\{1\}$, and the $(n+1)$st through $m$th components of $x$ are $0$.
So $f(x)=(\chi^a_{K_1}-\chi^a_{K_2})(\pi^a_{K_1})=1$.
Hence $H_1(\mathcal{V}_G)\neq 0$.
\end{proof}

\section{Finding a graphical RAAG presentation for $\pso$.}\label{s:presentation}

We now give a right-angled Artin presentation for $\pso$ when all support graphs are forests. We will be working with outer automorphism classes of elements throughout, however for ease of reading we suppress the bracket notation and write elements as $\pi^a_K \in \pso$ rather than $[\pi_K^a]$. 

\subsection{An alternative generating set for $\pso$}

Throughout this section we suppose that each support graph $\Delta_a$ is a forest with $k_a+1$ maximal subtrees (connected components) $C^a_0,\ldots,C^a_{k_a}$.
Since the vertices of $\Delta_a$ represent connected components of $\G-\st(a)$, and $\Delta_a$ has its own connected components, we usually refer to the connected components of $\Delta_a$ as maximal subtrees to avoid confusing repetition of the term ``component".
We pick a basepoint $x^a_i$ in each tree $C^a_i$. 
We say that $x_0^a$ is the \emph{preferred basepoint} of the forest $\Delta_a$. 

We need two kinds of generators for our generating set for $\pso$. We use a set of partial conjugations that are not necessarily standard generators 
We have already introduced the first kind.
Suppose $C=C_i^a$ is a maximal subtree of $\Delta_a$. 
Recall that $\zg{a}{C}$ denotes be the product 
\[\zg{a}{C}=\prod_{K \in C} \pi_K^a
\] 
 over all elements $K$ of the vertex set of $C$ (each $K$ is a connected component of $\G-\st(a)$).
These elements are central in $\pso$ by Proposition~\ref{p:r2}. 

We also introduce an element $\eg{a}{e}$ associated to each edge in $\Delta_a$. 
\begin{definition} Let $e$ be an edge in a maximal subtree $C$ of $\Delta_a$ with basepoint $x \in C$. 
The interior of the edge $e$ separates $C$ into two pieces. 
Let $\mathcal{L}$ be the component of $C - e$ which does not contain the basepoint $x$. 
We define $$\eg{a}{e} =\prod_{K \in \mathcal{L} }\pi_K^a
$$
\end{definition}
The choice of basepoint gives a uniform way of choosing a component of $C - e$, however this choice does not matter too much, at least in terms of commuting elements in $\pso$:
\begin{lemma}\label{l:changeofbasepoint}
Let $\mathcal{L}'$ be the component of $C-e$ which contains the basepoint of $C$, and let $$(\eg{a}{e})'=\prod_{K\in \mathcal{L'}} \pi_K^a.$$ 
Then $\eg{a}{e}(\eg{a}{e})'$ is central in $\pso$. 
In particular an element commutes with $\eg{a}{e}$ if and only if it commutes with $(\eg{a}{e})'$.
\end{lemma}

\begin{proof}
We simply observe that, as $V(\mathcal{L}) \cup V(\mathcal{L}')=V(C)$:
\[\eg{a}{e}(\eg{a}{e})'=\zg{a}{C},\]
which is central, by Proposition~\ref{p:r2}.
Since $(\eg{a}{e})'$ is the product of $\eg{a}{e}$ with a central element, we see that anything that commutes with $\eg{a}{e}$ also commutes with $(\eg{a}{e})'$.
By symmetry, they have exactly the same centralizers.
\end{proof}

Although elements of the form $\eg{a}{e}$ are not central in $\pso$, there is quite a strong requirement for a commutator of the form $[\eg{a}{e},\eg{b}{f}]$ to be nonzero.
\begin{proposition} \label{l:r1}
Let $e$ and $f$ be edges of $\Delta_a$ and $\Delta_b$ respectively. Then $\eg{a}{e}$ and $\eg{b}{f}$ commute unless:
\begin{itemize}
\item $(a,b)$ is an SIL pair and
\item  the edges $e$ and $f$ are of the form $\{[b]_a,L\}$ and $\{[a]_b, L\}$, where $L$ is a shared component of $(a,b)$.
\end{itemize}
\end{proposition}

\begin{proof}
Suppose that $[\eg{a}{e},\eg{b}{f}]\neq1$.  
If $(a,b)$ do not form an SIL-pair then $\eg{a}{e}$ and $\eg{b}{f}$ commute as all standard generators of the form $\pi_K^a$ and $\pi_L^b$ commute. 
We may therefore assume that $(a,b)$ is an SIL-pair. 
Let $C$ be the maximal subtree of $\Delta_a$ containing the dominating component $[b]_a$, and let $D$ be the maximal subtree of $\Delta_b$ containing $[a]_b$. 
The elements $\eg{a}{e}$ and $\eg{b}{f}$ will commute unless $e \in C$ and $f \in D$, as otherwise one of the products $\eg{a}{e}$ or $\eg{b}{f}$ will consist of standard generators only corresponding to subordinate components for $(a,b)$. 
Let $C' \subset C$ be the star of $[b]_a$ in $\Delta_a$; so $C'$ contains $[b]_a$ together with all the shared components of $(a,b)$.
Suppose for contradiction that $e$ is not an edge of $C'$;
then one component $\mathcal{L}$ or $\mathcal{L'}$ of $\Delta_a -e$ is disjoint from $C'$ and contains only vertices of subordinate components. 
Hence $\eg{b}{f}$ commutes with either $\eg{a}{e}$ or $(\eg{a}{e})'$, 
so by Lemma~\ref{l:changeofbasepoint}, it commutes with $\eg{a}{e}$. 
This contradicts our hypothesis, so $e$ must be an edge of $C'$.
The same argument applies with the location of $f$ in $D$. 
It follows that both $e$ and $f$ are of the form $\{[b]_a,L\}$ and $\{[a]_b, L'\}$ respectively, where $L$ and $L'$ are shared components for $(a,b)$. 
Lemma~\ref{l:changeofbasepoint} allows us to assume that the component of $C - e$ (respectively $D-f$) which does not contain the basepoint is the one containing $L$ (respectively $L'$), so that 
\begin{equation*} \eg{a}{e}=\pi_L^a \prod_{K} \pi_K^a \text{ and } \eg{b}{f}=\pi_{L'}^b \prod_{K'} \pi_{K'}^b \end{equation*}
where each $\pi_K^a$ (respectively $\pi^b_{K'}$) in the product is subordinate for the pair $(a,b)$.  
As partial conjugations along distinct shared components commute, it follows that $L=L'$ when $[\eg{a}{e},\eg{b}{f}]\neq1$.
\end{proof}

\subsection{The right-angled Artin presentation}

We are now in a position to give an explicit right-angled Artin presentation for the group $\pso$.
\begin{definition}
Let $A_\Theta$ be the right-angled Artin group with defining graph $\Theta$ given by vertices of the form:
\begin{itemize}
\item $v^a_e$ for each vertex $a \in \G$ and each edge $e$ in $\Delta_a$.
\item $v^a_{C}$ for each vertex $a \in \G$ and each maximal subtree $C$ of $\Delta_a$ not equal to the tree $C_0^a$ containing the preferred basepoint.
\end{itemize}
The graph $\Theta$ is given the following edges:
\begin{itemize}
\item There is an edge between each vertex $v^a_{C}$ and every other vertex in $\Theta$.
\item There is an edge between $v^a_e$ and $v^b_f$ unless $(a,b)$ forms an SIL-pair and $e=\{[b]_a,L\}$ and $f=\{[a]_b,L\}$ for some shared component $L$ of $(a,b)$.
\end{itemize}
\end{definition}

Note that the definition of $A_\Theta$ depends on the location of the preferred basepoints but is independent of the remaining basepoints. Propositions \ref{p:r2} and \ref{l:r1} immediately imply the following:

\begin{proposition}
The map on generators given by $\phi(v^a_C)=\zg{a}{C}$ and $\phi(v^a_e)=\eg{a}{e}$ induces a homomorphism $\phi\colon A_\Theta \to \pso$.
\end{proposition}

\subsection{Constructing an inverse map}

To show that $\pso \cong A_\Theta$ we  will construct an inverse map $\psi\colon\pso \to A_\Theta$. The first step is to write each standard generator $\pi_K^a$ as a product of elements of the form $\zg{a}{C}$ and $\eg{a}{e}$.

\begin{lemma}\label{l:gens}
Let $\pi^a_K$ be a standard generator of $\pso$. 
Let $C_0,\ldots,C_k$ be the set of maximal trees in the support graph $\Delta_a$, and let $e_0,\ldots,e_n$ be the edges of the support graph adjacent to $K$. 

\begin{enumerate} 
\item If $K$ is not the basepoint of its subtree in $\Delta_a$ and $e_0$ is the edge adjacent to $K$ in the direction of the basepoint then  
$$\pi_K^a=\eg{a}{e_0} (\eg{a}{e_1})^{-1} \cdots (\eg{a}{e_n})^{-1}.$$
\item If $K$ is the basepoint of some tree $C\neq C_0$  then 
$$\pi^a_K=\zg{a}{C} (\eg{a}{e_0})^{-1} (\eg{a}{e_1})^{-1}\cdots (\eg{a}{e_n})^{-1}.$$
\item If $K$ is the preferred basepoint of $\Delta_a$, then:
$$\pi^a_K=(\zg{a}{C_1})^{-1}\cdots (\zg{a}{C_k})^{-1}\cdot (\eg{a}{e_0})^{-1}(\eg{a}{e_1})^{-1} \cdots (\eg{a}{e_n})^{-1}.$$
\end{enumerate}

\end{lemma}

\begin{proof} 
We explain the proof of item (1).
Let $C$ be the maximal subtree of $\Delta_a$ containing $K$ and let $\mathcal{L}_i$ be the component of $C- e_i$ disjoint from the basepoint. 
The vertex set of $\mathcal{L}_0$ is the disjoint union of $\{K\}$ with the vertex sets of the $\mathcal{L}_i$ for $i\geq1$. 
Equation (1) then follows from the definition of $\eg{a}{e_i}$.   
A similar calculation applies to cases (2) and (3).
\end{proof}

\begin{corollary}
The homomorphism $\phi\colon A_\Theta \to \pso$ is surjective.
\end{corollary}

\begin{proof}
In Lemma~\ref{l:gens}, we wrote each element of the standard generating set as a product of elements in the image of $\phi$.
\end{proof}

Lemma~\ref{l:gens} gives an obvious candidate for an inverse map.

\begin{definition} 
With $e_0,\ldots,e_n$ and $C_0,\ldots,C_k$ as in Lemma~\ref{l:gens}, let: 
\begin{equation*}
\psi(\pi^a_K)=  
\begin{cases} v^a_{e_0} {v^a_{e_1}}^{-1} \cdots {v^a_{e_n}}^{-1} & \text{if $K$ is not a basepoint of $\Delta_a$} \\
v^a_C {v^a_{e_0}}^{-1} \cdots {v^a_{e_n}}^{-1} &\text{if $K$ is a basepoint but not preferred}\\
{v^a_{C_1}}^{-1}\cdots{v^a_{C_k}}^{-1}\cdot {v^a_{e_0}}^{-1} \cdots {v^a_{e_n}}^{-1} &\text{if $K$ is the preferred basepoint of $\Delta_a$} 
\end{cases}
\end{equation*}
This defines a map $\psi\colon \pso\to A_\Theta$.
\end{definition}

We owe the reader a proof that this map, as defined on generators, extends to a well defined homomorphism.
The following lemma reduces the number of cases which we need to run through:

\begin{lemma}\label{l:decomp}
Let $\pi^a_K$ be a standard generator of $\pso$ and let $e_0,\ldots,e_n$ be the edges in $\Delta_a$ adjacent to $K$. If $K$ is not a basepoint of its tree in $\Delta_a$ we assume that $e_0$ is the edge in the direction of the basepoint. 
There exists a central element $g \in A_\Theta$ such that
$$\psi(\pi^a_K)=g\cdot (v^a_{e_0})^\epsilon (v^a_{e_1})^{-1}\cdots (v^a_{e_n})^{-1},$$
where $\epsilon\in\{1,-1\}$. If $K$ is a basepoint in $\Delta_a$ then $\epsilon=1$, otherwise $\epsilon=-1$.
\end{lemma}

\begin{proof}
This follows from the definition of $\psi$ and the fact that each element $v^a_C$ is central in $A_\Theta$.
\end{proof}

\begin{lemma}\label{l:comm2}
Suppose that $[\psi(\pi^a_K),\psi(\pi^b_L)]\neq1$. Then $(a,b)$ forms an SIL-pair and either:
\begin{itemize}
\item $K$ and $L$ are both dominating for the pair $(a,b)$, or
\item $K$ is dominating for the pair $(a,b)$ and $L$ is shared, or
\item $L$ is dominating for the pair $(a,b)$ and $K$ is shared, or
\item $K=L$ is a shared component for the pair $(a,b)$.
\end{itemize}
\end{lemma}

\begin{proof} Let \begin{equation*} \psi(\pi^a_K)=g. (v^a_{e_0})^\epsilon (v^a_{e_1})^{-1}\cdots (v^a_{e_n})^{-1} \text{ and } \psi(\pi^b_L)=g'. (v^b_{f_0})^{\epsilon'} (v^b_{f_1})^{-1}\cdots (v^b_{f_m})^{-1} \end{equation*} be the decompositions of $\psi(\pi^a_K)$ and $\psi(\pi^b_L)$ respectively given by Lemma~\ref{l:decomp}. If these two elements do not commute in $A_\Theta$, then as $g$ and $g'$ are central, the elements \begin{equation*} (v^a_{e_0})^\epsilon (v^a_{e_1})^{-1}\cdots (v^a_{e_n})^{-1} \text{ and } (v^b_{f_0})^{\epsilon'} (v^b_{f_1})^{-1}\cdots (v^b_{f_m})^{-1} \end{equation*}
also do not commute in $A_\Theta$. In particular there exist $i$ and $j$ such that $v_{e_i}^a$ and $v_{f_j}^b$ do not commute in $A_\Theta$. From the definition of $A_\Theta$, this implies that $(a,b)$ is an SIL-pair and $e_i=\{[b]_a,L'\}$, $f_j=\{[a]_b,L'\}$ for some shared component $L'$ of $(a,b)$. As $K$ is an endpoint of $e_i$ and $L$ is an endpoint of $f_j$, one of the four cases listed above must hold. 
\end{proof}

\begin{proposition}
The map $\psi\colon \pso \to A_\Theta$ as defined on generators extends to a well defined homomorphism.
\end{proposition}

\begin{proof}
We need to check the relations (R1)--(R5) in Corollary~\ref{c:pres} are sent to the identity under the induced map from the free group on the standard generators of $\pso$ to $A_\Theta$. 
The relations in (R1)--(R3) are commutators $[\pi_K^a,\pi_L^b]$ corresponding to the following situations:
\begin{itemize}
\item The commutator $[a,b]=1$, so in particular $(a,b)$ is not an SIL-pair.
\item The components $K$ and $L$ are disjoint and non-dominating for the pair $(a,b)$.
\item One component is dominating and the other is subordinate for the pair $(a,b)$.
\end{itemize}
By Lemma~\ref{l:comm2}, in all three situations we have $[\psi(\pi^a_K),\psi(\pi^b_L)]=1$.

 The relations in (R4) are of the form $[\pi^a_K\pi_L^a,\pi^b_L]$, where $(a,b)$ is an SIL-pair, $K$ is dominating, and $L$ is shared for the pair $(a,b)$. Let $e$ be the edge $e=\{K,L\}$ in the support graph $\Delta_a$. Let $x$ be the basepoint of the component of $\Delta_a$ containing $K$ and $L$. If $x$ is closer to $K$ (including $x=K$) then $v^a_e$ occurs with exponent $-1$ in the decomposition of $\psi(\pi_K^a)$ and exponent $+1$ in the decomposition of $\psi(\pi_L^b)$. Similarly, if $x$ is closer to $L$ (including $x=L$) then  $v^a_e$ occurs with exponent $-1$ in the decomposition of $\psi(\pi_L^a)$ and with exponent $+1$ in the decomposition of $\psi(\pi_K^a)$. In either case, this term cancels out in any reduced word representing the product $\psi(\pi_K^a)\psi(\pi_L^a)$. More precisely, one can show that there exists a central element $g$ in $A_\Theta$ such that: $$\psi(\pi_K^a)\psi(\pi_L^a)=g\prod_{e_i} (v_{e_i}^a)^{-1},$$ 
where this product is taken over all edges $e_i\neq e$ in $\Delta_a$ adjacent to either $K$ or $L$. In contrast, the image of $\pi_L^b$ in $A_\Theta$ is of the form $$\psi(\pi^b_L)=g'. (v^b_{f_0})^\epsilon (v^b_{f_1})^{-1}\cdots (v^b_{f_m})^{-1},$$ where $g'$ is central and $f_0,\ldots, f_m$ are the edges adjacent to $L$ in $\Delta_b$. As $L$ is shared, the Star Lemma tells us that one of these edges $f_i=\{L,[a]_b\}$ has the dominating component $[a]_b$ as its other vertex, and the remaining edges are of the form $\{L,L'\}$, where $L'$ is a subordinate component of $(a,b)$ (this uses the fact that $\Delta_a$ is a forest). 
As $v_e^a$ does not occur in our decomposition of $\psi(\pi_K^a)\psi(\pi_L^a)$, it follows that $[v^a_{e_i},v^b_{f_j}]=1$ in $A_\Theta$ for all $e_i$ and $f_i$ in the above two decompositions. Hence $\psi(\pi_K^a)\psi(\pi_L^a)$ and $\psi(\pi_L^b)$ commute in $A_\Theta$.

Finally, each relation in (R5) is of the form $\prod_{K\in \Delta_a} \pi_K^a$. We want to show that $$\prod_{K \in \Delta_a}\psi(\pi_K^a)=1.$$ From the definition of $\psi(\pi_K^a)$, one can check that each element $v_e^a$ occurs with exponents $+1$ and $-1$ exactly once each in the above product (corresponding to the images of the generators given by the endpoints of the edge $e$ under $\psi$). Similarly, if $C_0,\ldots, C_k$ are the components of $\Delta_a$ and $i\geq1$ then $v^a_{C_i}$ also occurs with exponents $+1$ and $-1$ exactly once (the exponent $+1$ appears in the image of the element given by the basepoint of $C_i$, and the exponent $-1$ appears in the image of the generator corresponding to our preferred basepoint). As all the above elements commute in $A_\Theta$, it follows that $\prod_{K \in \Delta_a}\psi(\pi_K^a)=1$.  
\end{proof}

\begin{theorem}
The group $\pso$ is isomorphic to $A_\Theta$.
\end{theorem}

\begin{proof} It only remains to show that $\phi$ and $\psi$ are mutual inverses. The fact that $\phi(\psi(\pi_K^a)) = \pi_K^a$ follows directly from the definitions and Lemma~\ref{l:gens}. If $v^a_C$ is a generator of $A_\Theta$ then $$\psi(\phi(v^a_C))=\prod_{K\in C} \psi(\pi^a_K)$$
The element $v^a_C$ occurs exactly once with exponent $+1$ under the image of the standard generator $\pi_x^a$ corresponding to the basepoint of $C$. If $e\in C$, the generator $v_e^a$ occurs twice in the above product (corresponding to the two endpoints of $e$), once with exponent $+1$ and once with exponent $-1$. As all these elements commute, $\psi(\phi(v^a_C))=v^a_C$. Similarly $$\psi(\phi(v^a_e))=\prod_{K\in\mathcal{L}} \psi(\pi_K^a),$$ where $\mathcal{L}$ is the component of $C - e$ which does not contain the basepoint. The element $v^a_e$ occurs once in this product with exponent $+1$ in the image of the standard generator given by the one endpoint of $e$ which is contained in $\mathcal{L}$. Every edge $e' \in \mathcal{L}$ then occurs twice, once with exponent $+1$ and once with exponent $-1$. These appear in the images of the generators corresponding to the endpoints of $e'$, both of which lie in $\mathcal{L}$. It follows that $\psi(\phi(v^a_e))=v^a_e$. Hence the compositions $\phi \circ \psi$ and $\psi \circ \phi$ are the identity maps on $\pso$ and $A_\Theta$ respectively, so that both maps are isomorphisms.
\end{proof}

\bibliography{Version2}

\begin{thebibliography}{10}

\bibitem{AMP}
Javier Aramayona and Conchita Mart{\'{\i}}nez-P{\'e}rez.
\newblock On the first cohomology of automorphism groups of graph groups.
\newblock {\em J. Algebra}, 452:17--41, 2016.

\bibitem{BNS}
Robert Bieri, Walter~D. Neumann, and Ralph Strebel.
\newblock A geometric invariant of discrete groups.
\newblock {\em Invent. Math.}, 90(3):451--477, 1987.

\bibitem{BF}
Corey Bregman and Neil~J. Fullarton.
\newblock Infinite groups acting faithfully on the outer automorphism group of
  a right-angled artin group.
\newblock ArXiv preprint 1507.04359, to appear in \emph{Michigan Mathematical
  Journal}, 2015.

\bibitem{BV}
Martin~R. Bridson and Karen Vogtmann.
\newblock The {D}ehn functions of {$Out(F_n)$} and {$Aut(F_n)$}.
\newblock {\em Ann. Inst. Fourier (Grenoble)}, 62(5):1811--1817, 2012.

\bibitem{Brown}
Kenneth~S. Brown.
\newblock Trees, valuations, and the {B}ieri-{N}eumann-{S}trebel invariant.
\newblock {\em Invent. Math.}, 90(3):479--504, 1987.

\bibitem{CL}
Christopher~H. Cashen and Gilbert Levitt.
\newblock Mapping tori of free group automorphisms, and the
  {B}ieri--{N}eumann--{S}trebel invariant of graphs of groups.
\newblock {\em J. Group Theory}, 19(2):191--216, 2016.

\bibitem{CF}
Ruth Charney and Michael Farber.
\newblock Random groups arising as graph products.
\newblock {\em Algebr. Geom. Topol.}, 12(2):979--995, 2012.

\bibitem{CRSV}
Ruth Charney, Kim Ruane, Nathaniel Stambaugh, and Anna Vijayan.
\newblock The automorphism group of a graph product with no {SIL}.
\newblock {\em Illinois J. Math.}, 54(1):249--262, 2010.

\bibitem{D}
Matthew~B. Day.
\newblock Finiteness of outer automorphism groups of random right-angled
  {A}rtin groups.
\newblock {\em Algebr. Geom. Topol.}, 12(3):1553--1583, 2012.

\bibitem{FP}
Edward Formanek and Claudio Procesi.
\newblock The automorphism group of a free group is not linear.
\newblock {\em J. Algebra}, 149, 1992.

\bibitem{MR2888948}
Mauricio Gutierrez, Adam Piggott, and Kim Ruane.
\newblock On the automorphisms of a graph product of abelian groups.
\newblock {\em Groups Geom. Dyn.}, 6(1):125--153, 2012.

\bibitem{Hatcher}
Allen Hatcher.
\newblock {\em Algebraic topology}.
\newblock Cambridge University Press, Cambridge, New York, 2002.

\bibitem{KMM}
Nic Koban, Jon McCammond, and John Meier.
\newblock The {BNS}-invariant for the pure braid groups.
\newblock {\em Groups Geom. Dyn.}, 9(3):665--682, 2015.

\bibitem{KP}
Nic Koban and Adam Piggott.
\newblock The {B}ieri-{N}eumann-{S}trebel invariant of the pure symmetric
  automorphisms of a right-angled {A}rtin group.
\newblock {\em Illinois J. Math.}, 58(1):27--41, 2014.

\bibitem{MVW}
John Meier and Leonard VanWyk.
\newblock The {B}ieri-{N}eumann-{S}trebel invariants for graph groups.
\newblock {\em Proc. London Math. Soc. (3)}, 71(2):263--280, 1995.

\bibitem{Servatius}
Herman Servatius.
\newblock Automorphisms of graph groups.
\newblock {\em J. Algebra}, 126(1):34--60, 1989.

\bibitem{T}
William~P. Thurston.
\newblock A norm for the homology of {$3$}-manifolds.
\newblock {\em Mem. Amer. Math. Soc.}, 59(339):i--vi and 99--130, 1986.

\bibitem{Toinet}
Emmanuel Toinet.
\newblock A finitely presented subgroup of the automorphism group of a
  right-angled {A}rtin group.
\newblock {\em J. Group Theory}, 15(6):811--822, 2012.

\end{thebibliography}
\bibliographystyle{plain}

\end{document}